\documentclass[onecolumn]{autart} %onecolumn,twocolumn

\usepackage{graphicx,color}%
\usepackage{amsmath}
\usepackage{amssymb}

\newtheorem{theorem}{Theorem}[section]
\newtheorem{corollary}[theorem]{Corollary}
\newtheorem{lemma}[theorem]{Lemma}
\newtheorem{proposition}[theorem]{Proposition}

\newtheorem{example}{Example}
\newtheorem{remark}{Remark}[section]
\newtheorem{assumption}{Assumption}

\newcommand{\until}[1]{\{1,\dots, #1\}}
\newcommand{\subscr}[2]{#1_{\textup{#2}}}

\newcommand{\setdef}[2]{\{#1 \; | \; #2\}}

\newcommand{\map}[3]{#1: #2 \rightarrow #3}

\newcommand{\eps}{\varepsilon}

\newcommand{\naturals}{\mathbb{N}} % natural numbers
\newcommand{\integers}{\mathbb{Z}} % integer numbers
\newcommand{\reals}{\mathbb{R}} % real numbers
\newcommand{\integernonnegative}{\ensuremath{\mathbb{Z}}_{\ge 0}}
\newcommand{\real}{\ensuremath{\mathbb{R}}}

\newcommand{\integer}{{\mathbb{Z}}}
\newcommand{\Z}{\mathbb{Z}}

\newcommand{\R}{\reals}

\renewcommand{\Pr}{\mathbb{P}} % probability (is already defined in some package)
\newcommand{\Exp}{\mathbb{E}} % expectation

\newcommand{\G}{\mathcal{G}}    %for graphs
\newcommand{\E}{\mathcal{E}}    %for graphs
\newcommand{\V}{\mathcal{V}}    %for graphs
\newcommand{\xave}{\subscr{x}{ave}}%\newcommand{\xave}{x_{ave}}
\newcommand{\nave}{\subscr{n}{ave}}
\newcommand{\Tcon}{\subscr{T}{con}}
\newcommand{\Tall}{\subscr{T}{all}}
\newcommand{\1}{\mathbf{1}} % vector of ones
\newcommand{\diag}{\operatorname{diag}} % diagonal part
\newcommand{\trace}{\operatorname{tr}}

\newcommand{\A}{{\mathcal{A}}}       % operators or spaces

\newcommand{\I}{{\mathcal{I}}}
\renewcommand{\S}{{\mathcal{S}}}
\newcommand{\RR}{{\mathcal{R}}}
\newcommand{\NomPart}{{\mathcal{N}}}
\newcommand{\QuantPart}{{\mathcal{Q}}}

\newcommand{\qd}{q_{d}} %quantizers
\newcommand{\qdtilde}{{\tilde q}_{d}} %quantizers
\newcommand{\qp}{q_{p}}

\newcommand{\gii}{g_{1}}
\newcommand{\giii}{g_{2}}
\newcommand{\giv}{g_{3}}
\newcommand{\gvi}{g_{4}}
\newcommand{\gvii}{g_{5}}

\begin{document}
%\maketitle
\begin{frontmatter}
%\runtitle{Insert a suggested running title}
\title{%Consensus algorithms using quantized gossip communications
Gossip consensus algorithms via quantized communication\thanksref{footnoteinfo}}
\thanks[footnoteinfo]{This paper was not presented at any IFAC meeting. Corresponding author P.~Frasca Tel.~+39-011-5647552.
% A preliminary partial version of the present paper has appeared in \cite{PF-RC-FF-SZ:08cdc}.
The authors want to thank the anonymous reviewers for their constructive comments.
}

\author[SB]{Ruggero Carli}\ead{carlirug@engineering.ucsb.edu},
\author[Poli]{Fabio Fagnani}\ead{fabio.fagnani@polito.it.},
\author[Poli]{Paolo Frasca}\ead{paolo.frasca@polito.it},
\author[DEI]{Sandro Zampieri}\ead{zampi@unipd.it},

\address[SB]{Center for Control, Dynamical Systems and Computation, University of
    California, Santa Barbara, CA 93106, USA}
\address[Poli]{Dipartimento di Matematica, Politecnico
    di Torino, C.so Duca degli Abruzzi,~24, 10129 Torino, Italy}
\address[DEI]{DEI, Universit\`a di Padova, Via~Gradenigo~6/a, 35131 Padova, Italy}
%\address[IAC]{Universit\`a di Salerno, and Istituto per le Applicazioni del Calcolo-CNR, Roma, Italy}


\begin{abstract}
This paper considers the average consensus problem on a network of digital links, and proposes a set of algorithms based on pairwise ``gossip'' communications and updates. We study the convergence properties of such algorithms with the goal of answering two design questions, arising from the literature: whether the agents should encode their communication by a deterministic or a randomized quantizer, and whether they should use, and how, exact information regarding their own states in the update.
\end{abstract}
\end{frontmatter}

%\tableofcontents
\section{Introduction}
In the latest years, algorithms to solve consensus problems have attracted a lot of interest. In a consensus problem a group of agents has to agree about a certain quantity, starting from different initial estimates.  A special interest is devoted to average consensus, where the agents are requested to agree on the average of their initial estimates. Among the vast literature, we refer the reader to \cite{ROS-JAF-RMM:07} and references therein.
The difficulty of the problem resides in the communication constraints which are given to the agents. Such communication constraints are usually represented by a graph: nodes are agents and edges are available communication links. Moreover, the communication across the links can be assumed to be perfect, or rather be digital and possibly subject to bandwidth constraints, interferences, erasures, packet losses, noise, delays. Among the many algorithms for consensus proposed in the literature, particularly interesting is the so called gossip algorithm: at every time instant a randomly chosen pair of agents communicates and they average their states. Such algorithm, studied in detail in \cite{SB-AG-BP-DS:06}, has many appealing features: it reduces the number of communications with respect to deterministic algorithms and avoids data collision. The present paper is devoted to the adaptation of the gossip algorithm to a network of digital lossless channels, that is, subject to quantized communication.

\subsection*{Related works}
The constraint of quantization, due to the use of digital channels or to computing and memory constraints, has been considered in consensus problems in several recent papers \cite{LX-SB-SL:05,RC-FF-PF-TT-SZ:07,SK-JM:sub,AK-TB-RS:07,RC-FB-SZ:08j,TCA-MJC-MGR:08,PF-RC-FF-SZ:08,AN-AO-AO-JNT:08,MZ-SM:08a}.
The present work differs from previous ones on many respects. With the notable exception of \cite{AK-TB-RS:07}, and of the recent conference \cite{MZ-SM:08a}, quantization has not been investigated in the context of gossip algorithm. On the other hand, the work \cite{AK-TB-RS:07}, which has been an important reference to us, deals with a strictly related but different problem, that is, consensus of agents having quantized {\em states}, while our interest is in quantized communication among real-valued agents. The case of gossip algorithm is worth of consideration in the context of quantized communication, since it can not be reduced to the case of time-invariant communication. For instance, gossiping introduces randomness in the algorithm. This allows using probabilistic tools in the analysis, and can significantly improve the convergence properties (in a probabilistic sense). Indeed, it prevents the onset of periodic dynamics, as those noticed in the fixed-topology communication since \cite{RC-FF-PF-TT-SZ:07}.

One could believe that the finite-bandwidth communication constraint can be easily dealt with using incremental and adaptive quantizers, in which the finite-length messages encode the last update rather than the agents state.
Actually, such approach to consensus algorithm with quantized communication has been undertaken in \cite{RC-FB-SZ:08j}, assuming a static communication network and logarithmic or zooming-in/zooming-out quantizers. However, in a gossip scenario the random time-dependence of the active communication links makes difficult to design a similar algorithm: hence the analysis of static quantizers is valuable. For these reasons, in this paper we shall consider static uniform quantizers.

\subsection*{Statement of contributions}
The goal of this paper is to analyze the effects of quantized communication on the gossip algorithm \cite{SB-AG-BP-DS:06}, or, equivalently, the opportunity of using a gossip communication when quantization of the messages is imposed. The agents' states are assumed to be real numbers, while the transmitted messages belong to a finite set. %are integer numbers.
Through both analytical results and simulations, we investigate two design questions: whether the agents should use the deterministic or the probabilistic quantizer, and whether they should use, and how, exact information regarding their own states in the update. This is done in the following way: on one hand, we compare a {\em deterministic} uniform quantizer and a {\em probabilistic} uniform quantizer; on the other hand, we consider three different update rules ({\em partially quantized}, {\em totally quantized}, {\em compensating}), which differ in how the agents use the information about their own state. Both the quantizers and the update rules are introduced in Section~\ref{sec:Statement}.

Our results, which describe the limit behavior of the algorithms, are obtained by two different techniques. For the compensating rule with probabilistic quantization, we give a mean squared error analysis and convergence is proved as time goes to infinity. In all the other cases, we study a suitable Markov chain symbolic dynamics, obtaining results of convergence in finite time. Such a fact is remarkable, since it underlines the discrete nature of the problem, in spite of the state space being continuous.
In more detail, we can summarize our results as follows. In Section~\ref{sec:GlobQA}, we show that the totally quantized rule ensures that the consensus is reached almost surely in a finite time, both using the deterministic quantizer and the probabilistic quantizer. The drawback of this update rule is that it does not preserve the average of the initial conditions, and the deviation from the initial average happens to be possibly large if the deterministic quantizer is used. On the other hand, the compensating rule preserves the initial average at each iteration of the algorithm, but does not guarantee that consensus is reached. However, we prove in Section~\ref{sec:PartQA} that the states get as close to average consensus as the size of quantization step, in a finite time. Finally, in Section~\ref{sec:Naive} we consider the partially quantized rule. While it does not preserve the initial average, it has good convergence properties. Indeed, if the deterministic quantizer is used, the states get in a finite time as close to consensus as the size of the quantization step. If the probabilistic quantizer is used, we can argue a stronger result of asymptotical convergence to consensus, and the expectation of the consensus value is the average of the initial states.

%%%%%%%%%%%%%%%%%%%%%%%%%%%%%%%%%%%%%%%%%%%%%%%%%%%%%%%%%%%%%%%%%%%%%%%%%%%%%%%%%%%%%%%%%%%%%%%%%%%%%%%%%%%%%%%%%%%%
\section{Problem statement}\label{sec:Statement}
We start recalling the gossip average consensus algorithm: at every time step, a randomly chosen pair of agents communicates, and they average their states. This algorithm, brought to wide audience by \cite{SB-AG-BP-DS:06}, has many appealing features: it reduces the number of communications with respect to deterministic algorithms and avoids data collisions.
Let us describe such algorithm in more detail.
Assume we are given an undirected connected graph $\G=(\V,\E)$, with $\V=\until{N}$ and ${\E}\subset \{(i,j): i, j \in \V\}$. Each of the nodes of the graph is referred to as an {\em agent}, and endowed with a state, which is scalar function of time, $x_i(t),$ for $i\in \V$. The values $x_i(0)$ are given. At each time step $t\in \integernonnegative$, one edge $(i,j)$ is randomly selected in $\E$ with positive probability $W^{(i,j)}$ such that $\sum_{(i,j)\,\in\,\E}W^{(i,j)}=1$. Let $W$ be the matrix with entries $W_{ij}=W^{(i,j)}$. The two agents connected by the selected edge average their states according to
\begin{align}\label{eq:GossipStandard}
x_i(t+1)&= \frac{1}{2} x_i(t)+\frac{1}{2} x_j(t)\nonumber\\
x_j(t+1)&= \frac{1}{2} x_j(t)+\frac{1}{2} x_i(t)
\end{align}
while\begin{align}\label{eq:GossipStandardN}
x_h(t+1)=x_h(t) \qquad \mbox{ if $h\neq i,j$. }
\end{align}
Let $E_{ij}=(e_i-e_j)(e_i-e_j)^*$
and\footnote{The symbol $M^*$ is used to denote the conjugate transpose of the matrix $M$. }
\begin{equation}\label{eq:defPt}
P(t)=I-\frac{1}{2}E_{ij}
\end{equation}
where $e_i=[0,\ldots,0,1,0,\ldots,0]^*$ is a $N\times1$ unit vector with the $i$-th component equal to $1$. Then
(\ref{eq:GossipStandard}) and (\ref{eq:GossipStandardN}) can be
written in a vector form as
\begin{equation}\label{eq:GossipStandardVect}
x(t+1)=P(t)x(t),
\end{equation}
where $x(t)=[x_1(t),\ldots,x_N(t)]^*.$
Note that $P(t)$ is a symmetric doubly stochastic matrix, and then \eqref{eq:GossipStandardVect} preserves the average of states. It is known~\cite{SB-AG-BP-DS:06,ATS-AJ:08} that, if the graph $\G$ is connected and each edge $(i,j)\,\in\,\E$ is selected with a strictly positive probability, then, for every initial condition $x(0)$, the algorithm $(\ref{eq:GossipStandardVect})$ almost surely reaches the {\em average consensus}. That is, almost surely $$\lim_{t\rightarrow \infty}x(t)=\xave(0)\1,$$ where $\1$ is the column vector whose entries are 1, and for $t\in \integernonnegative$, we define $\xave(t)=N^{-1}\sum_{i=0}^{N}x_i(t).$ If instead the weaker condition holds that $\lim_{t\rightarrow \infty}x(t)=\xi\1,$ for some $\xi\in \real$, we say that the algorithm reaches {\em consensus}.

Note that the gossip algorithm \eqref{eq:GossipStandardVect} relies upon a crucial assumption: each agent transmits to its neighboring agents the precise value of its state. Instead, in this paper we assume that the communication network is constituted of digital links. This prevents the agents from having a precise knowledge about the state of the other agents. Indeed, through a digital channel, the $i$-th agent can only send to its neighbors symbolic
data: using  this data, the neighbors of the $i$-th agent can build an estimate of the $i$-th agent's
state. We denote this estimate by ${\hat x}_i(t)$, and  let $\hat{x}(t)=\left[{\hat x}_1(t), \ldots, {\hat x}_N(t) \right]^*.$ In this paper, the estimate is simply the received symbol, computed via a suitable quantizer, that is an application mapping real numbers into a discrete set.

\subsection{Quantizers}
The first design issue is how to quantize the states, that is how to map the continuous space of states into a discrete alphabet of messages. Given $X\subseteq\reals^\nu$, we call quantizer a map $\map{q}{X}{S}$, where $S$ is a finite or countable set, endowed with the discrete topology. If we have a vector $x\in X^N$, with a slight abuse of notation, we will use the notation $q(x)\in S^N$ to denote the vector such that $q(x)_i=q(x_i)$.
Many quantizers have been proposed in the vast literature on the subject \cite{GNN-FF-SZ-RJE:07}: here we concentrate on uniform quantizers with a countable alphabet, which can be thought as maps $\map{q}{\real}{\integer}$, up to a suitable rescaling. Two of such quantizers are of special interest to us, which we define below, and we call the deterministic and the probabilistic quantizer. A broader discussion is given in the final section.
The {\em deterministic quantizer} is defined as follows.
Let $\map{\qd }{\R}{\Z} $ be the map which sends $z\in \R$ into its nearest integer, namely,
\begin{equation*}%\label{eq:qd}
\qd (z)=n\in\Z\quad\Leftrightarrow
\begin{array}{c}
  \quad z\,\,\in\,\,[n-1/2,n+1/2[,\:\mbox{ if } z\geq 0\\
  \quad z\,\,\in\,\,]n-1/2,n+1/2],\:\mbox{ if } z< 0.\\
\end{array}
\end{equation*}
This map enjoys the property that, for all $z\in \R$, it holds $|z-\qd(z)|\le \frac{1}{2}.$

The {\em probabilistic quantizer} $\map{\qp }{\R}{\Z}$ is defined as follows.\footnote{Elsewhere in the literature \cite{TCA-MJC-MGR:08,SK-JM:sub,JJX-ZQL:05} this quantizer has been introduced as a result of {\em dithering}, that is the addition of a small random noise before (deterministic) quantization. } For any $x\,\in\, \reals$, the image $\qp(x)$ is a random variable on $\integers$ defined by
\begin{equation}\label{eq:qp}
\qp(x)=\left\{
\begin{array}{rcl}
\lfloor x \rfloor & \mbox{with probability} & \lceil x \rceil-x\\
\lceil x \rceil & \mbox{with probability} & x-\lfloor x \rfloor.\\
\end{array}
\right.
\end{equation}
In this case,  for all $z\in \R$, it holds that $|z-\qp(z)|\le 1.$
The following lemma states two important properties of the probabilistic quantizer.
\begin{lemma}\label{lem:PropQuantProb}
For every $x\,\in\,\reals$, it holds that
\begin{align*}%\label{eq:PropQuantProb1}
&\Exp\left[\qp(x)\right]=x,\\
&\Exp\left[\left(x-\qp(x)\right)^2\right]\leq \frac{1}{4}.%\label{eq:PropQuantProb2}.
\end{align*}
\end{lemma}
\begin{proof}
The first equation is immediate, and the second one follows from computing
$$\Exp\left[\left(x-\qp(x)\right)^2\right]=x \lfloor x\rfloor+ x \lceil x\rceil-\lfloor x\rfloor \lceil x\rceil -x^2\le \frac{1}{4}.$$
%Equation~\eqref{eq:PropQuantProb1} is immediate and \eqref{eq:PropQuantProb2} follows from computing
%$$\Exp\left[\left(x-\qp(x)\right)^2\right]=x \lfloor x\rfloor+ x \lceil x\rceil-\lfloor x\rfloor \lceil x\rceil -x^2\le \frac{1}{4}.$$
\end{proof}

The quantizers defined above map ${\reals}$ into ${\integers}$, and have quantization bins of length $1$. More general uniform quantizers, having as quantization step a positive real number $\eps$, can be obtained from $\map{q}{\reals}{\integers}$ by defining $q^{(\eps)}(x)= \eps q(x/\eps)$. Hence, the general case can be simply recovered by a suitable scaling, and our choice is not restrictive.

\subsection{Update rules}
We introduce {\em three updating rules} for the states, which require quantized communication. To describe them, let us assume that $(i,j)$ be the edge selected at the $t$-th iteration.

\begin{itemize}
\item[a)] {\em (Totally quantized)}
In the first strategy, $i$ and $j$, in order to update their states, use only their estimates, as follows,
\begin{align}\label{eq:GossipQuant1}
x_i(t+1)&= \frac{1}{2} \hat{x}_i(t)+\frac{1}{2} \hat{x}_j(t)\nonumber\\
x_j(t+1)&= \frac{1}{2} \hat{x}_j(t)+\frac{1}{2} \hat{x}_i(t),
\end{align}
or, equivalently in vector form, by recalling the definition of $P(t)$ given in \eqref{eq:defPt},
\begin{equation}\label{eq:GossipQuantVect1}
x(t+1)=P(t)\hat{x}(t).
\end{equation}

\item[b)] {\em (Partially quantized)}
If instead the agents have access to their real-valued state, a natural update rule is
\begin{align}\label{eq:GossipQuant3}
x_i(t+1)&= \frac{1}{2} x_i(t)+\frac{1}{2} \hat{x}_j(t)\nonumber\\
x_j(t+1)&= \frac{1}{2} x_j(t)+\frac{1}{2} \hat{x}_i(t),
\end{align}
or, equivalently in vector form,
\begin{equation}\label{eq:GossipQuantVect3}
x(t+1)=\frac{1}{2}x(t)+(P(t)-\diag P(t))\hat{x}(t).
\end{equation}

\item[c)] {\em (Compensating)}
Both the above update rules do not preserve the average of states, which can be a significant drawback in some applications. To cope with this problem, we propose a third update rule, in which agents use both their real-valued states {\em and} their quantized values. To understand the idea behind this, consider the standard gossip update \eqref{eq:GossipStandard}, which can be rewritten as
\begin{align}\label{}
x_i(t+1)&= x_i(t)-\frac{1}{2} x_i(t)+\frac{1}{2} x_j(t)\nonumber\\
x_j(t+1)&= x_j(t)-\frac{1}{2} x_j(t)+\frac{1}{2} x_i(t).\nonumber
\end{align}
We then propose the following updating rule,
\begin{align}\label{eq:GossipQuant2}
x_i(t+1)&= x_i(t)-\frac{1}{2} \hat{x}_i(t)+\frac{1}{2} \hat{x}_j(t)\nonumber\\
x_j(t+1)&= x_j(t)-\frac{1}{2} \hat{x}_j(t)+\frac{1}{2} \hat{x}_i(t),
\end{align}
or, equivalently in vector form,
\begin{equation}\label{eq:GossipQuantVect2}
x(t+1)=x(t)+(P(t)-I)\hat{x}(t).
\end{equation}
\end{itemize}
Note that in facts law~\eqref{eq:GossipQuantVect2} preserves the initial state average, as the law~\eqref{eq:GossipStandardVect} does. Formally, defining $\xave(t)=\frac{1}{N}\1^*x(t)$, we have that $\xave(t)=\xave(0)$, for all $t\geq 0$. Indeed, $ \1^*x(t+1)=\1^*x(t)+\1^*(P(t)-I){\hat x}(t)=\1^*x(t)$, where the last equality follows from the fact that, since $P(t)$ is doubly stochastic for all $t\geq 0$, then $\1^*(P(t)-I)=0$ for all $t\geq 0$.
The idea of the agents updating their own state using both exact and quantized information about their own state at the previous time step, has already been shown useful in the case of quantized consensus on a time-independent network \cite{PF-RC-FF-SZ:08}. From the point of view of communication theory, this strategy is meant to fully exploit the implicit channel feedback which comes from quantization: since communication is quantized, but the channels are reliable, each agent knows that its neighbors are going to receive the message it has transmitted.

In the sequel, we refer to (\ref{eq:GossipQuantVect1}) as the {\em totally quantized} rule, to \eqref{eq:GossipQuantVect3} as the {\em partially quantized} rule, and to (\ref{eq:GossipQuantVect2}) as the {\em compensating} rule.
In the following sections, we proceed with a detailed analysis of the dynamical systems induced by the above rules.

\section{Compensating update}\label{sec:PartQA}
We start our analysis from the update rule \eqref{eq:GossipQuantVect2}, considering first the case of deterministic, and then of probabilistic quantizers.
\subsection{Deterministic quantizer}
Consider the compensating strategy
\begin{align}\label{eq:GossipQuant2d}
x_i(t+1)&= x_i(t)-\frac{1}{2} \qd(x_i(t))+\frac{1}{2} \qd(x_j(t))\nonumber\\
x_j(t+1)&= x_j(t)-\frac{1}{2} \qd(x_j(t))+\frac{1}{2} \qd(x_i(t)).
\end{align}

The limit behavior of the above algorithm can be studied exploiting a natural {\it symbolic dynamics}
interpretation of the states dynamics, obtaining results which are reminiscent of those in \cite{AK-TB-RS:07}.
We define $n_i(t)=\lfloor 2 x_i(t)
\rfloor$ for all $i\in \V$, and let $n(t)=\left[n_1(t), \ldots,
n_N(t)\right]^*$.

To start, we need the following technical lemma, whose proof can be found in \cite{PF-RC-FF-SZ:08}.
\begin{lemma}\label{lem:FormulasFloor}
Given $\alpha,\beta\, \in \naturals$ and $x\in \reals$, it
holds \begin{align*} %\label{formula floor 2}
&\lfloor
x\rfloor=\left\lfloor
\frac{\lfloor\alpha x\rfloor}{\alpha}\right\rfloor,\\
%\label{formula floor 1}
&q_d(x)=\lfloor
x+1/2\rfloor=\left\lceil\frac{1}{2}\left\lfloor \frac{\lfloor 2
\beta x \rfloor}{\beta}\right\rfloor \right\rceil.
\end{align*}
\end{lemma}
Lemma~\ref{lem:FormulasFloor} implies that $\qd(x_{i}(t))=\left\lceil\frac{n_i(t)}{2}\right\rceil$,
and then we can manipulate the dynamics as follows:
\begin{align*}
x_i(t+1)&=  x_i(t)-\frac{1}{2} \qd(x_i(t))+\frac{1}{2} \qd(x_j(t))\\
\lfloor 2 x_i(t+1)\rfloor&=  \lfloor 2 x_i(t)\rfloor- \qd(x_i(t))+ \qd(x_j(t)),
\end{align*}
From this we obtain that
\begin{align*}
n_i(t+1)&= n_i(t)-\left\lceil\frac{n_i(t)}{2}\right\rceil+\left\lceil\frac{n_j(t)}{2}\right\rceil\\
&= \left\lfloor\frac{n_i(t)}{2}\right\rfloor+\left\lceil\frac{n_j(t)}{2}\right\rceil.
\end{align*}
We have thus found an iterative system involving only the symbolic
signals $n_i(t)$. When the edge $(i,j)$ is selected, $i$ and $j$ adjourn their states following the pair dynamics
\begin{equation}\label{evolution-automaton-p}
\left(n_i(t+1),n_j(t+1)\right)=\gii(n_i(t), n_{j}(t))\\
\end{equation}
where $\map{\gii}{\integers\times \integers}{\integers\times \integers}$ is
$$
\gii(h,k)=\left(\left\lfloor\frac{h}{2}\right\rfloor+\left\lceil\frac{k}{2}\right\rceil,
\left\lfloor\frac{k}{2}\right\rfloor+\left\lceil\frac{h}{2}\right\rceil\right).
$$
Notice that $\gii$ is symmetric in its arguments, in the sense that if $\gii(h,k)=(\eta,\chi)$, then $\gii(k,h)=(\chi,\eta)$. Since $n_i(t)=\lfloor 2x_i(t) \rfloor$, the analysis of the evolution of \eqref{evolution-automaton-p} allows us to obtain information about the asymptotics of $x_i(t)$.
We have  the following result.
\begin{theorem}\label{theo:ConvInM}
Let $n(t)$ evolve according to (\ref{evolution-automaton-p}), and let
\begin{equation}\label{eq:SetR6}
\RR=\left\{r\,\in\,\integers^N\,: \exists \, \alpha \,\in\,\integers
\mbox { s. t. }\,r-\alpha\1\,\in\,\left\{0, 1\right\}^N\right\}.
\end{equation}
Then, for every initial condition $n(0)\in \integers$, almost surely there exists $\Tcon\,\in\,\integernonnegative$ such that $n(t)\, \in\,\RR$ for all $t \geq \Tcon.$
\end{theorem}
\begin{proof}
The proof is based %\cite{JRN:97}
on verifying the following three facts:
\begin{itemize}
\item [(i)] the set $\RR$, defined in~\eqref{eq:SetR6}, is
an invariant subset for the evolution described
by~\eqref{evolution-automaton-p}; \item [(ii)] $n(t)$ is a Markov
process on a finite number of states; \item [(iii)] starting from any state in $\Z^N$, there is a
positive probability for $n(t)$ to reach a state in $\RR$
in a finite number of steps.
\end{itemize}
%Standard results on Markov chains \cite{JRN:97} ensure that, if the above three facts yield true, the thesis is proven.
Let us now check them in order.
\begin{itemize}
\item [(i)] Let $h\,\in\,\integers$. Observe that
$$
\gii\left(h, h+1\right)=\left\{
\begin{array}{l}
(h+1, h) \qquad \text{if $h$ is even}\\
(h, h+1) \qquad \text{if $h$ is odd.}
\end{array}
\right.
$$
This implies that $\RR$ is an invariant subset for the
dynamics described by~\eqref{evolution-automaton-p}.

\item [(ii)] Markovianity immediately follows from the fact that
subsequent random choices of the edges are independent. We prove
now that the states are finite. To this aim let $(h', k')=\gii(h, k).$
The form of $\gii$ implies that
$$
\max \left\{h', k'\right\}\leq \max \left\{h, k\right\} \qquad
\min \left\{h, k\right\}\leq \min\left\{h', k'\right\}.
$$
Define
\begin{equation}\label{eq:m-M}
m(t)=\min_{1\leq i \leq N}n_i(t) \qquad \text{and} \qquad
M(t)=\max_{1 \leq i \leq N}n_i(t).
\end{equation}
Then, the above inequalities imply that $m(t)\geq m(0)$ and $M(t)\leq M(0)$. This implies that the cardinality of the set of the states is upper bounded by $(M(0)-m(0)+1)^N.$

\item[(iii)] Let $D(t)=M(t)-m(t).$ The proof of (iii) is based on the following strong result about the monotonicity of $D(t)$: if $D(t)\geq 2$, then there exists $\tau\,\in\,\naturals$ such that
\begin{equation}\label{eq:ProD}
\Pr[D(t+\tau)<D(t)]>0.
\end{equation}
Now we prove~\eqref{eq:ProD}.
We define $\I_a(t)=\{i \in \V: n_i(t)=a\}$, and fix a time $t_0\in \naturals$. We shall prove that the cardinality of $\I_{m(t_0)}(t)$, denoted by $|\I_{m(t_0)}(t)|$, does not increase as a function of time, and that, if $D(t)\geq 2$, then there is a positive probability that it decreases within a finite number of
time steps. Notice first that, for $h,\,k\,\in\,\integers$,
$\gii(h+2,k+2)=\gii(h,k)+2.$ Hence, by an appropriate translation of the
initial condition, we can always restrict ourselves to the case
$m(t_0)\,\in\,\left\{0,1\right\},$ which is easier to handle.

{\it Case} $m(t_0)=0$. In this case it is possible for a nonzero
state to decrease to $0$, but only in the case of a swap between
$0$ and $1$. This assures that $|\I_{m(t_0)}(t)|$ is
nonincreasing. Let $\S(t)$ denote the set of nodes which
have value $m(t_0)+2$ or larger. Since $D(t_0)\geq 2$ then
$\S(t_0)$ is non empty. Now let
$(v_1,v_2,\ldots,v_{p-1},v_p)$ be a shortest path between
$\I_{m(t_0)}(t)$ and $\S(t_0)$. Such a  path exists since
$\G$ is connected. Note that $v_1\,\in\,\I_{m(t_0)}(t)$
and $v_p\,\in\,\S(t_0)$ and that
$\left\{v_2,\ldots,v_{p-1}\right\}$ could be an empty set; in this
case a shortest path between $\I_{m(t_0)}(t)$ and $\S(t_0)$
has length $1$. Note also that all the nodes in the path, except $v_1$ and $v_p$, have value $1$ at time $t_0$, otherwise
$(v_1,v_2,\ldots,v_{p-1},v_p)$ would not be a shortest path. Since
each edge of the communication graph has a positive probability of
being selected at any time, there is also a positive probability
that in the $p-1$ time units following $t_0$ the edges of this path
are selected sequentially, starting with the edge $(v_1,v_2)$. At
the last step of this sequence we have that the values of
$v_{p-1}$ and $v_p$ are updated. By observing again, that the pair of values $(0,1)$ is transformed by (\ref{evolution-automaton-p}) into the pair $(1,0)$, we have that the value of $v_{p-1}$, when
the edge $(v_{p-1},v_p)$ is selected, is equal to $0$. This update, for the form of $(\ref{evolution-automaton-p})$, makes the value of both nodes be strictly greater than $0$.
Therefore, this proves that
$|\I_{m(t_0)}(t_0+p-1)|<|\I_{m(t_0)}(t_0)|$ with positive probability.
Clearly, if $|\I_{m(t_0)}(t_0)|=1$ then we have also that
$D(t_0+p-1)<D(t_0)$ with
positive probability.

{\it Case} $m(t_0)=1$. In this case no state can decrease to $1$,
and thus $|\I_{m(t_0)}(t)|$ is again nonincreasing. Let
$\I_{m(t_0)}(t)$, $\S(t)$ and $(v_1,v_2,\ldots,
v_{p-1},v_p)$ be defined as in the previous case. In
this case all the nodes $v_2,\ldots,v_{p-1}$ in the path have
value equal to $2$. Moreover observe that also the sequence of
edges $(v_{p-1},v_p)$, $(v_{p-2},v_{p-1})$, \ldots,$(v_2,v_3)$,
$(v_1,v_2)$ has positive probability of being selected in the
$p-1$ time units following $t_0$. At the last step of this sequence
of edges, the values of $v_1$ and $v_2$ are updated. The value of $v_1$ is equal to $1$. Since the value of $v_p$ at time $t_0$ is greater or
equal to $3$, and since the pair $(2,3)$ is transformed by (\ref{evolution-automaton-p})
into $(3,2)$, we have that the value of $v_2$ when the edge $(v_1,v_2)$ is selected, is
greater or equal to $3$. This update, for $(\ref{evolution-automaton-p})$, causes the value
of both nodes to be strictly greater than $1$. Hence $|\I_{m(t_0)}(t_0+p-1)|<|\I_{m(t_0)}(t)|$
with positive probability. Again, if $|\I_{m(t_0)}(t)|=1$ then we have also that $D(t_0+p-1)<D(t)$
with positive probability.\\
Consider now the following sequence of times $t=t_0,t_1,t_2,\ldots$. For each $i\geq 0$,
%if $|\I_{m(t_i)}(t_i)|>1$, then we
let $t_{i+1}$ be the first time for which there is a
positive probability that $|\I_{m(t_i)}(t_{i+1})|< |\I_{m(t_i)}(t_{i})|.$ Let now $k\,\in\,\integernonnegative$
be such that $|\I_{m(t_k)}(t_{k})|=1.$ Then we have that $D(t_{k+1})<D(t_k)$ with positive probability. This ensures the validity of~\eqref{eq:ProD}.

The proof of the fact (iii) follows directly from~\eqref{eq:ProD}. Indeed, let $\bar n \, \notin\,\RR$,
then, from a repeated application of~\eqref{eq:ProD} it follows that, there exists a path of the Markov chain connecting $\bar n$ to a state $\bar n'=\left[\bar{n}'_1,\ldots,\bar{n}'_N\right]$, such that $\max\left\{\bar{n}'_1,\ldots,\bar{n}'_N\right\}-\min\left\{\bar{n}'_1,\ldots,\bar{n}'_N\right\}<2,$ that is, $\bar{n}'\,\in\,\RR$.
\end{itemize}
%This proves the thesis.
\end{proof}

We can go back to the original system, and prove the following result.
\begin{corollary}\label{corol:SymbConv}
Consider the algorithm \eqref{eq:GossipQuant2d}. Then, almost surely, there exists $\Tcon\in\integernonnegative$ such that for all $t\ge \Tcon$,
\begin{equation}\label{eq:conv-det-comp}
   {|x_i(t)-x_j(t)|}\leq 1\quad \forall\,i,j.
\end{equation}
and hence,
$$
\|x(t)-\xave(0)\1 \|_\infty\leq 1.
$$
\end{corollary}
\begin{proof}
The proof is an immediate consequence of Theorem~\ref{theo:ConvInM} and of the relation
$n_i(t)=\left\lfloor2x_i(t)\right\rfloor$, which assure that the
states belong to two consecutive quantization bins.
\end{proof}
Our results about algorithms for quantized communication (here and in following sections) can naturally be related to the analysis of gossip consensus algorithms with quantized {\em states}. Namely, the important connections with the work in \cite{AK-TB-RS:07} are the object of the next two remarks.
\begin{remark}\normalfont
It is worth noting that Theorem~\ref{theo:ConvInM} is a  variation
of Lemma~3 and Theorem~1 in \cite{AK-TB-RS:07}. In
\cite{AK-TB-RS:07} the authors introduced a class of {\it
quantized gossip algorithms}, satisfying the following
assumptions. Let $(i,j)$ be the edge selected at time $t$ and let
$n_i(t)$ and $n_j(t)$  the values at time $t$ of node $i$ and of
node $j$ respectively. If $n_i(t)=n_j(t)$ then $n_i(t+1)=n_i(t)$
and $n_j(t+1)=n_j(t)$. Otherwise, defined
$D_{ij}=|n_i(t)-n_j(t)|$,  the method used to update the values
has to satisfy the following three properties:
\begin{itemize}
\item [(P1)] $n_i(t+1)+n_j(t+1)=n_i(t)+n_j(t)$,
\item [(P2)] if $D_{ij}(t)>1$ then $D_{ij}(t+1)<D_{ij}(t)$, and
\item [(P3)] if $D_{ij}(t)=1$ and (without loss of generality) $n_i(t)<n_j(t)$, then $n_i(t+1)=n_j(t)$ and $n_j(t+1)=n_i(t)$.
Such update is called {\em swap}.
\end{itemize}
Now we substitute the property $(P3)$ either with the property
\begin{itemize}
\item [(P3')] if $D_{ij}(t)=1$ and (without loss of generality)
$n_i(t)<n_j(t)$, then, if $n_i(t)$ is odd, then $n_i(t+1)=n_j(t)$
and $n_j(t+1)=n_i(t)$, otherwise if $n_i(t)$ is even then
$n_i(t+1)=n_i(t)$ and $n_j(t+1)=n_j(t)$
\end{itemize}
or with the property
\begin{itemize}
\item [(P3'')] if $D_{ij}(t)=1$ and (without loss of generality)
$n_i(t)<n_j(t)$, then, if $n_i(t)$ is even then $n_i(t+1)=n_j(t)$
and $n_j(t+1)=n_i(t)$, otherwise if $n_i(t)$ is odd then
$n_i(t+1)=n_i(t)$ and $n_j(t+1)=n_j(t)$.
\end{itemize}

If we consider the class of algorithms satisfying (P1), (P2),
(P3') or satisfying (P1), (P2), (P3''), it is possible to prove
that Lemma~3 and Theorem~1 stated in \cite{AK-TB-RS:07} hold true
also for this class. The proofs are analogous to that of  Theorem~\ref{theo:ConvInM} provided above. Moreover it is easy to see that the algorithm (\ref{evolution-automaton-p}) satisfies the
properties (P1), (P2), (P3'). This represents an alternative way
to prove Theorem~\ref{theo:ConvInM}.
\end{remark}

\begin{remark}\normalfont
We observe that a sensible algorithm for consensus with {\em deterministically} quantized communication comes applying the algorithm of Kashyap, Ba{\c s}ar and Srikant \cite{AK-TB-RS:07}, which we refer to as the KBS algorithm. The adaptation we propose is as follows. The initial real values are first quantized, and then the KBS algorithm is applied: convergence to consensus is guaranteed up to an error of one, in finite time. Since the worst case error committed by quantizing the initial states is $1/2$, we conclude that for $t$ large enough $\|x(t)-\xave(0)\1\|_{\infty}\le \frac{3}{2}.$ Hence the convergence properties of this algorithm, which renounces making any use of the continuity of states, are worse, but comparable with those of the algorithm \eqref{eq:GossipQuant2d}.
%
%This fact is illustrated in Figure~\ref{fig:Basar} where we provide a comparison between the compensating strategy via deterministic quantizers and the algorithm KBS. Precisely, we plotted the behavior of $d(t)$ for both strategies on the same graph considered in Figure~\ref{fig:PartiallyDnonConv}.
\end{remark}

%\begin{figure}[ht]
%\begin{center}
%%\includegraphics[height=6cm]{FigureGossip/BasarVersusDeterministicPartiallyWorst}
%\includegraphics[height=6cm]{Figures/BasarVersusDeterministicPartiallyWorst}
%  \caption{Plot of $d(t)$, as in \eqref{eq:d6}, for the compensating strategy with deterministic quantization, and for the algorithm {\em KBS} proposed in \cite{AK-TB-RS:07}.}%: worst case of 50 algorithm runs.}
%\label{fig:Basar}\end{center}
%\end{figure}

\subsection{Probabilistic quantizer}\label{sec:PartQA-qp}
In this section we assume that the information exchanged between
the systems is quantized by means of the probabilistic quantizer
$\qp$ described in \eqref{eq:qp}, namely
$\hat{x}_i(t)=\qp(x_i(t))$. Dealing with updates based on the probabilistic quantizer, we make the following natural assumption of statistical independence.
\begin{assumption}\label{ass:Assindip}
For all $t\in \integernonnegative$, given the values $x_i(t)$ for all $i\in \V$, the random variables $\qp(x_i(t))$, as $i$ varies, form an independent set. Moreover, for all $t\in \integernonnegative$, given the value $x_i(t)$ for any $i\in\V$, the random variable $\qp(x_i(t))$ is independent from $x_j(t)$, for every $j\neq i$.
\end{assumption}

The algorithm for the compensating strategy, when the edge
$(i,j)$ is chosen, can be written as
\begin{align}\label{eq:GossipQuant2p}
x_i(t+1)&= x_i(t)-\frac{1}{2} \qp(x_i(t))+\frac{1}{2} \qp(x_j(t))\nonumber\\
x_j(t+1)&= x_j(t)-\frac{1}{2} \qp(x_j(t))+\frac{1}{2} \qp(x_i(t)).
\end{align}
%Similarly to the compensating strategy via deterministic quantizers
%(\ref{eq:GossipQuant2d}), also (\ref{eq:GossipQuant2p}) does not
%reach the consensus in general. Again we report a simulation
%showing this fact. In Figure~\ref{pPNC} the behavior of the
%quantity $d(t)$, defined in (\ref{eq:d6}), is depicted for the
%same connected random geometric graph considered in Figure~\ref{fig:PartiallyDnonConv}.
%\begin{figure}[!ht]
%\begin{center}
%%\includegraphics[width=\columnwidth]{Figures/partiallyProbNonConv.eps}
%\includegraphics[width=\columnwidth]{Figures/partiallyProbNonConv}
%  \caption{Behavior of $d$ for a connected random geometric graph with $N=50$. } \label{pPNC}\end{center}
%\end{figure}
%Note that the quantity $d(t)$ stays visibly away from $0$, meaning
%that the average consensus is not reached.

The analysis of (\ref{eq:GossipQuant2p}) is more complicated than
for the corresponding law (\ref{eq:GossipQuant2d}). This is mainly
due to the lack of convexity properties which were used in the
analysis of (\ref{eq:GossipQuant2d}). The following example shows
this type of difficulty.
\begin{example}\normalfont
Consider (\ref{eq:GossipQuant2d}) and assume that the edge $(i,j)$ has been selected at time $t$. Without loss of generality assume that $x_i(t)\leq x_j(t)$. Then, by convexity arguments, we have that $\lfloor x_i(t)\rfloor \leq x_i(t+1), x_j(t+1) \leq \lceil x_j(t)\rceil$. This is no longer true for (\ref{eq:GossipQuant2p}). As a numerical example assume that $x_i(t)=3.4$ and $x_j(t)=3.6$. Then with probability $0.16$ we have that $\qp(x_i(t))=4$ and $\qp(x_j(t))=3$. In such a case, by \eqref{eq:GossipQuant2p}, we have that $x_i(t+1)=2.9$ and $x_j(t+1)=4.1$. Hence, $x_i(t+1)$ and $x_j(t+1)$ do not belong to the interval $\left[\lfloor x_i(t)\rfloor, \lceil x_j(t)\rceil\right]$.
\end{example}
For this reason, we do not develop a symbolic analysis for this
algorithm, and we do not prove convergence in finite time. %By
%simulations we can see that (\ref{eq:GossipQuant2p}) does not
%drive the states of the systems inside the same bin of
%quantization, as the corresponding strategy
%(\ref{eq:GossipQuant2d}) using deterministic quantizers. In Figure~\ref{fig:PPD}, we depict the behavior of the quantity
%\begin{equation*}
%s(t)=\max_{1\,\leq\, i,j\, \leq \,N}|x_i(t)-x_j(t)|.
%\end{equation*}
%for the same random geometric graph considered in Figure~\ref{pPNC}. In this simulation we assume that the initial
%condition $x_i(0)$ is randomly chosen inside the interval $[-10,10]$.
%\begin{figure}[!ht]
%\begin{center}
%%\includegraphics[width=\columnwidth]{Figures/PartiallyProbDelta.eps}
%\includegraphics[width=\columnwidth]{Figures/PartiallyProbDelta}
%  \caption{Behavior of $s$ for a connected random geometric graph with $N=50$. } \label{fig:PPD}\end{center}
%\end{figure}
%Note that $s$ asymptotically oscillates around $2$.
Instead of a symbolic analysis, we provide a mean-square analysis, yielding interesting convergence results. We start by observing that (\ref{eq:GossipQuant2p}) can be rewritten as
\begin{equation}\label{eq:stateeqP}
x(t+1)=P(t)x(t)+(P(t)-I)\left(\qp(x(t))-x(t)\right)
\end{equation}
Let
\begin{equation*}%\label{eq:y6}
y(t)=\left(I-\frac{1}{N}\1\1^*\right)x(t),
\end{equation*}
and remark that $y(t)=x(t)-\frac{1}{N}\1\1^*x(0).$

Now, from \eqref{eq:stateeqP}, we can write
\begin{align*}
\left(I-\frac{1}{N}\1\1^*\right)x(t+1)&=\left(I-\frac{1}{N}\1\1^*\right)P(t)x(t)+
\left(I-\frac{1}{N}\1\1^*\right)\,(P(t)-I)\,\left(\qp( x(t))-x(t)\right).
\end{align*}
Define the quantization error as
$$
e(t)=\qp(x(t))-x(t).
$$
Since
$\left(I-\frac{1}{N}\1\1^*\right)P(t)=P(t)\left(I-\frac{1}{N}\1\1^*\right)$
and $\left(I-\frac{1}{N}\1\1^*\right)\left(P(t)-I\right)=P(t)-I$, we obtain the following recursive relation
in terms of the variables $e(t)$ and $y(t)$:
\begin{equation}\label{eq:stateeqY}
y(t+1)=P(t)y(t)+(P(t)-I)e(t).
\end{equation}
In order to perform an asymptotic analysis of (\ref{eq:stateeqY}), it is convenient to introduce the
following matrices. Let
\begin{equation}\label{eq:def-Sigmas}
\Sigma_{yy}(t)=\Exp\left[y(t)y^*(t)\right], \quad
\Sigma_{ee}(t)=\Exp\left[e(t)e(t)^*\right], \quad
\Sigma_{ye}(t)=\Exp\left[y(t)e(t)^*\right],
\end{equation}
where the expectation is taken with respect to both the randomness due to gossip communication and to quantization.
Equation $(\ref{eq:stateeqY})$ leads to the following recursive equation in terms of the above matrices
\begin{align}\label{eq:SigmaYY}
\Sigma_{yy}(t+1)&=\Exp\left[P(t)\Sigma_{yy}(t)P(t)\right]+\Exp\left[P(t)\Sigma_{ye}(t)(P(t)-I)\right]\nonumber\\
&\qquad\qquad+\Exp\left[(P(t)-I)\Sigma_{ye}^*P(t)\right]+\left(P(t)-I\right)\Sigma_{ee}(t)\left(P(t)-I\right).
\end{align}

The following proposition states some correlation properties of
the variables $y$ and $e$.
\begin{proposition}\label{prop:Properties_y_w}
Consider the vector random variables $y(t)$ and $e(t)$ defined above, and the matrices in \eqref{eq:def-Sigmas}. Then
\begin{equation*}%\label{eq:Property1-5}
\Exp\left[e(t)\right]=0 \qquad \text{and} \qquad \Sigma_{ee}(t)=\diag\left\{\sigma_1^2(t),\ldots, \sigma_N^2(t)\right\}
\end{equation*}
where $\sigma_i^2(t)=\Exp\left[e_i^2(t)\right]$ is such that $\sigma_i^2(t)\leq 1/4$, for all $i\in \until{N}$ and for all $t\geq 0$.\\
Moreover
\begin{equation}\label{Property2-5}
\Sigma_{ye}(t)=0,
\end{equation}
for all $t\geq 0$.
\end{proposition}
\begin{proof} Using conditional expectation properties, and Lemma~\ref{lem:PropQuantProb}, we have that
\begin{align}\label{eq:EqualityChain1-5}
\Exp\left[e_i(t)\right]&=\Exp\left[\,\Exp\left[\qp(x_i(t))-x_i(t)|x_i(t)\right]\right]\nonumber\\
&=\Exp\left[\,\Exp\left[\qp(x_i(t))|x_i(t)\right]-x_i(t)\right]\nonumber\\
&=\Exp\left[x_i(t)-x_i(t)\right]\nonumber\\
&=0.
\end{align}
Moreover, for $i\neq j$, using Assumption~\ref{ass:Assindip},
\begin{align*}%\label{eq:EqualityChain2-5}
\Exp\left[e_i(t)e_j(t)\right]&=\Exp\left[\Exp\left[e_i(t)e_j(t)\,|\,x_i(t),\,
x_j(t)\right]\right]\nonumber\\
&=\Exp\left[\Exp\left[e_i(t)\,|\,x_i(t),\,
x_j(t)\right]\Exp\left[e_j(t)\,|\,x_i(t),\,
x_j(t)\right]\right]\nonumber\\
&=\Exp\left[\Exp\left[e_i(t)\,|\,x_i(t)\right]\Exp\left[e_j(t)\,|\,
x_j(t)\right]\right]\nonumber\\
&=0
\end{align*}
%
%&=\Exp\left[\,e_i(t) \Exp \left[ \qp(x_j(t))-x_j(t)|x_i(t),x_j(t)\right] \right]\nonumber\\
%&=\Exp\left[\,e_i(t) \Exp \left[ \qp(x_j(t))-x_j(t)|x_j(t)\right] \right]\nonumber\\
%&=\Exp\left[\,e_i(t)\left(\Exp\left[\qp(x_j(t))|x_j(t)\right]- x_j(t)\right)\right]\nonumber\\
%&=\Exp\left[\,e_i(t)\left(x_j(t)-x_j(t)\right)\right]\nonumber\\
%&=0
%\end{align}
%where, both in (\ref{eq:EqualityChain1-5}) and in (\ref{eq:EqualityChain2-5}),  we have used the fact that by Lemma~\ref{lem:PropQuantProb},
%$\Exp\left[\qp(x_j(t))|x_j(t)\right]=x_j(t)$.
If $i=j$, using again Lemma~\ref{lem:PropQuantProb}, we have that
\begin{align*}%\label{eq:EqualityChain3-5}
\Exp\left[e_i^2(t)\right]&=\Exp\left[\left(\qp(x_i(t))-x_i(t)\right)^2\right]\nonumber\\
&=\Exp\left[\,\Exp \left[ \left(\qp(x_i(t))-x_i(t)\right)^2|x_i(t)\right] \right]\nonumber\\
&\leq\Exp\left[ \frac{1}{4} \right]\nonumber\\
&=\frac{1}{4}
\end{align*}
%where again by Lemma~\ref{lem:PropQuantProb}, we used the fact $\Exp \left[ \left(\qp(x_i(t))-x_i(t)\right)^2
%|x_i(t)\right]\leq 1/4$. Equations (\ref{eq:EqualityChain1-5}),
%(\ref{eq:EqualityChain2-5}) and
%(\ref{eq:EqualityChain3-5}) establish the validity of $(\ref{eq:Property1-5})$.\\
An argument similar to the one used to prove that
$\Exp\left[e_i(t)e_j(t)\right]=0$ allows to prove that
$\Exp\left[x_i(t)e_j(t)\right]=0$ for any $i\neq j$. This easily
yields (\ref{Property2-5}).
\end{proof}

From the above properties we have that (\ref{eq:SigmaYY}) can be rewritten as
\begin{equation}\label{eq:SigmaYYY}
\Sigma_{yy}(t+1)=\Exp\left[P(t)\Sigma_{yy}(t)P(t)\right]+\Exp\left[\left(P(t)-I\right)\Sigma_{ee}(t)\left(P(t)-I\right)\right].
\end{equation}
To estimate the asymptotic distance from the initial average, we introduce the cost function
\begin{equation}\label{eq:FuncCost}
J(W)=\limsup_{t\rightarrow \infty}\sqrt{\frac{1}{N}
\Exp[\|y(t)\|^2]}.
\end{equation}
The cost depends on the selection probabilities $W$, and, thanks to the above definitions, can be computed as
\begin{equation}\label{eq:cost}
J(W)=\limsup_{t\rightarrow \infty}\sqrt{\frac{1}{N}\trace\left\{\Sigma_{yy}(t)\right\}}.
\end{equation}

We can rewrite the evolution law (\ref{eq:SigmaYYY}) as
\begin{equation*}
\Sigma_{yy}(t+1)=\NomPart(\Sigma_{yy}(t))+\QuantPart(\Sigma_{ee}(t)),
\end{equation*}
where $\NomPart$ and $\QuantPart$ are linear operators from
$\reals^{N\times N}$ to itself. Namely, given a matrix $M$,
$\NomPart(M)=\Exp\left[P(t)M P(t)\right]$ and
$\QuantPart(M)=\Exp\left[\left(P(t)-I\right)M\left(P(t)-I\right)\right].$

It is useful to remark that $\NomPart$ is actually the evolution on $\Sigma_{yy}$ for the gossip algorithm
\cite{SB-AG-BP-DS:06}, in the absence of quantization error, while $\QuantPart$ can be regarded as a disturbance due to the quantization error. From \cite{FF-SZ:08}, we know that in the case of no quantization the system converges almost surely to consensus. This implies that $\NomPart$ is an asymptotically stable operator when restricted to the subspace ${\S}=\{M\in\reals^{N\times N}: \1^* M\1=0\}.$ Since $\1^* \QuantPart(M)\1=0$ for any matrix $M$ and $\Sigma_{yy}(0)\in \S$, we have that $\Sigma_{yy}(t)\in \S$ for all $t\ge 0$.

Computing $J(W)$ is a quite difficult problem. We then try to simplify the problem by introducing the following auxiliary system
\begin{equation}\label{eq:AuxiliarySystem}
\bar{\Sigma}(t+1)=\Exp\left[P(t)\bar{\Sigma}(t)P(t)\right]+\frac{1}{4}\Exp\left[\left(P(t)-I\right)^2\right],
\end{equation}
where $\bar{\Sigma}(0)=\Sigma_{yy}(0)$, and the following cost function
$$
\bar{J}= \limsup_{t\rightarrow \infty}\sqrt{\frac{1}{N}\trace\{\bar{\Sigma}(t)}\}.
$$
In principle, $\bar J$ should depend on $W$, too. However, we are
going to prove that this is not the case. We have the following
comparison result.
\begin{proposition}
Consider the cost functions $J(W)$ and $\bar J$. We have that
$$
J(W)\leq \bar{J}.
$$
\end{proposition}
\begin{proof}
To prove the statement we show, by induction on $t$, that $\bar{\Sigma}(t)\geq \Sigma_{yy}(t)$ for all $t\geq 0$, where the inequality is meant in matricial sense, that is, $\bar{\Sigma}(t)-\Sigma_{yy}(t)$ is a semidefinite positive matrix.

Since $\bar{\Sigma}(0)= \Sigma_{yy}(0)$, the assertion is true for $t=0$. Assume now that  $\bar{\Sigma}(t)\geq \Sigma_{yy}(t)$ is true for a generic $t$. Then,
\begin{align*}
\bar{\Sigma}(t+1)&-\Sigma_{yy}(t+1)\\
&=\Exp\left[P(t)\bar{\Sigma}(t)P(t)\right]+\frac{1}{4}\Exp\left[\left(P(t)-I\right)^2\right]-\left(\Exp\left[P(t)\Sigma_{yy}(t)P(t)\right]+\Exp\left[\left(P(t)-I\right)\Sigma_{ee}(t)\left(P(t)-I\right)\right]\right)\\
&=\Exp\left[P(t)(\bar{\Sigma}(t)-\Sigma_{yy}(t))P(t)\right]+\Exp\left[\left(P(t)-I\right)\left(\frac{1}{4}I-\Sigma_{ee}(t)\right)\left(P(t)-I\right)\right].
\end{align*}
Since by inductive hypothesis $\bar{\Sigma}(t)\geq\Sigma_{yy}(t)$ and since by Proposition~\ref{prop:Properties_y_w} we know that
 $\Sigma_{ee}(t)\leq  \frac{1}{4}I$ for all $t\geq 0$, we have that $\bar{\Sigma}(t+1)-\Sigma_{yy}(t+1)\geq0$.
\end{proof}
Observe now that, since $P(t)^2=P(t)$, we obtain that $\Exp[(I-P(t))^2]=I-\Exp[P(t)]$.
From this fact we obtain the following result.
\begin{proposition}\label{prop:limBarSigma}
Given the above definitions and (\ref{eq:SigmaYYY}),
$$
\lim_{t\to \infty}\bar{\Sigma}(t)=\frac{1}{4}\left(I-\frac{1}{N}\1\1^*\right).
$$
\end{proposition}
\begin{proof}
Define the matrix $\bar{\QuantPart}=\Exp[(I-P(t))^2]$. Since $\bar
\Sigma_{yy}(0)\in \S$, and $\NomPart$ is asymptotically stable if
restricted to the subspace $\S$, then $$\lim_{t\to
\infty}\bar \Sigma(t)=\sum_{t=0}^{+\infty} {\NomPart}^{(t)}(\bar \QuantPart).$$ This is the only fixed point of the iteration law
(\ref{eq:AuxiliarySystem}). Thus we are left to prove that
$\Sigma^*=\frac{1}{4}\left(I-\frac{1}{N}\1\1^*\right)$ is a fixed
point, that is $\Sigma^*={\NomPart}(\Sigma^*)+\bar {\QuantPart}$. This
is true, because
\begin{align*}
{\NomPart}(\Sigma^*)+\bar {\QuantPart}&=\frac{1}{4}\Exp\left[P(t)\left(I-\frac{1}{N}\1\1^*\right)P(t)\right]+
\frac{1}{4}(I-\Exp[P(t)])\\
&=\frac{1}{4}\left\{\Exp\left[P(t)^2\right]-\frac{1}{N}\1\1^*+I-\Exp[P(t)]\right\}\\
&=\frac{1}{4}\left\{I-\frac{1}{N}\1\1^*\right\}.
\end{align*}
\end{proof}

%\begin{remark}\normalfont
%In case the initial condition $x(0)$ be a random variable such that $\Exp[x(0)]=0$ and $\Exp\left[x(0)x^*(0)\right]=\sigma_0^2I$ for some $\sigma_0^2\,>\,0,$ it  follows that
%$\Sigma_{yy}(0)=\sigma_0^2\left(I-N^{-1}\,\1\1^*\right).$
%\end{remark}

\begin{corollary}
For every weight matrix $W$, it holds true that $J(W)\leq \frac{1}{2}.$
\end{corollary}
\begin{proof}
From Proposition~\ref{prop:limBarSigma} we can argue that $\bar J=\frac{1}{2}\sqrt{\frac{N-1}{N}}$, and since $J(W)\leq \bar J,$ the claim follows.
\end{proof}

From these theorems we draw a strong conclusion about the
convergence of the algorithm. In spite of missing consensus in the
strict sense, the asymptotical mean squared error of the algorithm
is smaller than the size of the quantization bin, and has a bound
which does not depend on the number of the agents, nor on the
topology of the graph, nor on the probability of the edges
selection.

%%%%%%%%%%%%%%%%%%%%%%%%%%%%%%%%%%%%%%%%%%%%%%%%%%%%%%%%%%%%%%%%%%%%%%%%%%%%%%%%%%%%%%%%%%%%%%%%%%%%%%%%%%%%%%%%%%%%%%%%
\section{Totally quantized update}\label{sec:GlobQA}
This section is devoted to study the totally quantized update
\begin{align*}
x_i(t+1)&= \frac{1}{2} \hat x_i(t)+\frac{1}{2} \hat x_j(t)\\
x_j(t+1)&= \frac{1}{2} \hat x_j(t)+\frac{1}{2} \hat x_i(t),
\end{align*}
introduced in \eqref{eq:GossipQuant1}, with both the deterministic and probabilistic quantizers.
\subsection{Deterministic quantizer}
In this subsection we consider the totally quantized strategy with deterministic quantizer
\begin{align}\label{eq:GossipQuant1d}
x_i(t+1)&= \frac{1}{2} \qd(x_i(t))+\frac{1}{2} \qd(x_j(t))\nonumber\\
x_j(t+1)&= \frac{1}{2} \qd(x_j(t))+\frac{1}{2} \qd(x_i(t)).
\end{align}
We underline immediately that the update~\eqref{eq:GossipQuant1d} only uses quantized information, and not exact information combined with quantized information as the update~\eqref{eq:GossipQuant2d}. This makes the analysis of~\eqref{eq:GossipQuant1d} slightly easier than the analysis of~\eqref{eq:GossipQuant2d}.
We show in this subsection that the law \eqref{eq:GossipQuant1d} drives, almost surely, the systems to consensus at an integer value. However, the initial average of states is not preserved in general. Again, the analysis of this algorithm can be performed efficiently by means of a symbolic dynamics.
Let again $n_i(t)=\lfloor 2 x_i(t) \rfloor$ for all $i\in \V$. From
(\ref{eq:GossipQuant1d}) and the fact that
$\qd(x_{i}(t))=\left\lceil\frac{n_i(t)}{2}\right\rceil,$
we obtain
\begin{equation}\label{eq:DinGloQuaStra}
(n_i(t+1), n_j(t+1))=\left(\giii(n_i(t), n_j(t)), \giii(n_i(t),
n_j(t)) \right)
\end{equation}
where $\map{\giii}{\integers\times \integers}{\integers}$ is defined as
$$
\giii(h,k)=\left\lceil\frac{h}{2}\right
\rceil+\left\lceil\frac{k}{2}\right \rceil.
$$
Define \begin{equation}\label{eq:SetR6bis}
\A=\setdef{y\in
\integers^N}{\exists \,\alpha\, \in\, \integers \mbox{ such that } y=2\alpha\1}.
\end{equation}

We have the following result.
\begin{theorem}\label{th:naivedet} Let $n(t)$ evolve according to
(\ref{eq:DinGloQuaStra}). For every fixed initial condition
$n(0)$, almost surely there exists $\Tcon\,\in\,\integernonnegative$ such that
$n(t)\, \in\,\A$ for all $t \geq \Tcon.$
\end{theorem}
\begin{proof}
As for the proof of Theorem~\ref{theo:ConvInM}, it is sufficient to verify the following three facts:
\begin{enumerate}
\item [(i)] each element in the set $\A$ is invariant
 for the evolution described by~\eqref{eq:DinGloQuaStra};
 \item [(ii)] $n(t)$ is a Markov process on a finite number of
states; \item [(iii)] starting from any state in $\Z^N$, there is a positive probability for $n(t)$
to reach a state in $\A$ in a finite number of steps.
\end{enumerate}

Let us now check them in order.
\begin{enumerate}
\item [(i)] is trivial.
\item [(ii)] Markovianity immediately follows from the fact that subsequent random choices of the edges are independent. To prove that the states are finite, define $m(t)$ and $M(t)$ as in~\eqref{eq:m-M}. Let $p, q\,\in\,\integers$ with $p\leq q$. Then, the form of $\giii$ implies that $p\leq \giii(p, q)\leq q+r_q$ where $r_q$ denotes the remainder in the division of $q$ by $2$. It follows that\begin{equation}\label{bound} m(0)\le n_i(t)\leq M(0)+r_{M(0)}\;\forall i \in \V\;\;\forall t\ge 0\,.\end{equation} This yields (ii).
\item[(iii)]  Let us fix $t=t_0$, and assume that $n(t_0)\notin \A$. We prove that there exists $\tau\,\in\,\naturals$ such that $\Pr\left[n(t_0+\tau)\,\in\,\A\right]>0$. We start by observing that, from the assumption of having a connected graph, there exists $(h,k) \in \E$ such that $n_h(t_0)=m(t_0)$, $n_k(t_0)=q$ and $\giii(m(t_0),q)>m(t_0)$. Indeed, two cases are given when $n(t_0)\notin \A$.
    \begin{itemize}
        \item If $m(t_0)<M(t_0)$, then it suffices to consider an edge $(h,k)$ such that $n_h(t_0)=m(t_0)$ and $n_k(t_0)=q>m(t_0)$, which gives $\giii(m(t_0),q)>m(t_0)$. Note that such an edge exists from the hypothesis of having a connected graph;
        \item if $m(t_0)=M(t_0)$, necessarily we have that $m(t_0)$ and $M(t_0)$ are odd; then $\giii(m(t_0),m(t_0))>m(t_0)$.
    \end{itemize}
    We define now $\I_a(t)=\{i \in \V: n_i(t)=a\}$.  The above discussion implies that $|\I_{m(t_0)}(t_0+1)|<|\I_{m(t_0)}(t_0)|$
    with the positive probability of choosing the edge $(h,k)$ and hence that there is also a positive probability that at some finite time $t'>t_0$, $|\I_{m(t_0)}(t')|=0$, that is $m(t')>m(t_0).$ Iterating this argument and recalling that $M(t)\leq M(t_0)+r_{M(t_0)}$ for all $t\geq t_0$, it follows that there exists $\tau\,\in\,\naturals$ such that $\Pr\left[n(t_0+\tau)\,\in\,\A\right]>0$.
\end{enumerate}
This proves the thesis.
\end{proof}

We can now go back to the original system. The following corollary follows immediately from the definition of $n(t)$.
\begin{corollary}\label{cor:ConvPartiallyD}
Let $x(t)$ evolve according to (\ref{eq:GossipQuant1d}). Then
almost surely there exists $\Tcon\in \integernonnegative$ and $\alpha\, \in\, \integers$
such that $x_i(t)=\alpha$ for all $i\in \V$ and for all
$t\ge \Tcon$.
\end{corollary}

We have already underlined the fact that this strategy does not preserve the initial average, in general. However, the convexity argument developed in the above proof, step (ii), implies a worst case bound on the error committed by the algorithm, as
\begin{align}\label{eq:BoundOnDeviation}
\nonumber|\alpha-\xave(0)|\le & \frac{1}{2}\left|2\alpha-\frac{1}{N}\1^*n(0)\right|+\frac{1}{2}\left|\frac{1}{N}\1^*n(0)-2\xave(0)\right|\\
\nonumber\le &\frac{1}{2} \big(M(0)-m(0)+1\big)+\frac{1}{2}\\
\le & \max_{i}{x_i(0)}-\min_{i}{x_i(0)}+\frac{3}{2}.
\end{align}
The significance of this apparently conservative bound will be discussed in Remark~\ref{rem:deviation}.

%In facts, the algorithm behaves much better than this conservative bound, as appears in the following simulation.
%In Figure~\ref{fig:GossipDeterTotallyCompleto} we plot the variable $z$ that
%is defined as follows. In the totally quantized strategy we have
%that, almost surely $\lim_{t\rightarrow \infty}=\alpha\1$ for some
%random integer $\alpha$. Let $z=|\alpha-1/N\1^*x(0)|.$ In words,
%$z$ represents the distance from the consensus point to which the
%totally quantized strategy leads the systems and the average of
%the initial condition. We have depicted the value of $z$ for a
%family of random geometric graphs \cite{MP:03} of increasing size from $N=10$
%up to $N=80$. The initial condition $x_i(0)$ is chosen randomly
%inside the interval $[-100,100]$ for all $1\leq i \leq N.$
%\begin{figure}[ht]
%\begin{center}
%%\includegraphics[width=\columnwidth]{Figures/DetGossipRandGeo.eps}
%\includegraphics[width=\columnwidth]{Figures/DetGossipRandGeo}
%  \caption{Behavior of $z$ for a family of random geometric graphs in case of
%deterministic quantizers and of totally quantized strategy.}
%\label{fig:GossipDeterTotallyCompleto}\end{center}
%\end{figure}
%Moreover for each $N$, $z$ is computed as the mean of $100$
%trials. We can see that the value of $z$ is increasing in $N$ and
%assumes values that are not negligible with respect to the
%quantization step size.

%%%%%%%%%%%%%%%%%%%%%%%%%%%%%%%%%%%%%%%%%%%%%%%%%%%%%%%%
\subsection{Probabilistic quantizer}
The algorithm for the totally quantized strategy, when the edge
$(i,j)$ is chosen, can be written as
\begin{align}\label{eq:GossipQuant1p}
x_i(t+1)&= \frac{1}{2} \qp(x_i(t))+\frac{1}{2} \qp(x_j(t))\nonumber\\
x_j(t+1)&= \frac{1}{2} \qp(x_j(t))+\frac{1}{2} \qp(x_i(t)).
\end{align}
A first remark is that $\Exp[\xave(t+1)=\xave(t)],$ that is, the average is preserved in expectation.
Below we prove that the law (\ref{eq:GossipQuant1p}), as the law (\ref{eq:GossipQuant1d}), drives almost surely the systems to consensus at an integer value. Using a probabilistic quantizer in a gossip algorithm, we have to deal with two sorts of randomness, since the interacting pair is randomly selected,
and the quantization map is itself random. This makes the analysis
of (\ref{eq:GossipQuant1p}) more complicated than the analysis of
(\ref{eq:GossipQuant1d}). However, again, we are able to prove the
convergence by a {\it symbolic dynamics} approach.

Let again $n_i(t)=\lfloor 2 x_i(t) \rfloor$ for all $i\in \V$, and let $n(t)=\left[n_1(t),\ldots,n_N(t)\right]^*$. Before finding a recursive equation for $n(t)$, we need to introduce the following random variable. Let
$$\Tall=\inf \setdef{t}{{\rm at\,\, time\,\, t\,\, every\,\,
 node\,\, in\,\,}\V {\rm\,\, has \,\,been\,\, selected\,\,
 at\,\, least\,\,once}}.$$
$\Tall$ is an integer random variable which is almost surely
finite, because nodes are selected with positive probability. Note
that, from~\eqref{eq:GossipQuant1p}, we have that $x_i(t)\,\in\,\left\{a,
a+1/2\right\}$ for some integer number $a$, for all $t\geq
\Tall$. This allows us to disregard the evolution before
$\Tall$ and to analyze, for $t\ge\Tall$, the symbolic dynamics
as follows. For $t\ge\Tall$, by recalling how the
probabilistic quantizer works, we have that
$$
\qp(x_i(t))=\left\{
\begin{array}{cc}
\frac{n_i(t)}{2} & \mbox{ if $n_i(t)$ is even} \\
\begin{array}{cc}
\lceil\frac{n_i(t)}{2}\rceil & \mbox{ with probability 1/2} \\
\lfloor\frac{n_i(t)}{2}\rfloor & \mbox{ with probability 1/2}
\end{array}
&\mbox{ if $n_i(t)$ is odd.}
\end{array}
\right.$$
Let $\xi_1$ and $\xi_2$ be two independent Bernoulli
random variables with parameter $1/2$ and define
$\map{\giv}{\integers\times\integers}{\integers}$ by
$$
\giv(h,k)=\left\lceil\frac{h}{2}\right
\rceil+\left\lceil\frac{k}{2}\right \rceil-\xi_1 r_h-\xi_2 r_k,
$$
where $r_h$ denotes the remainder of the division of $h$ by $2$.
If, at time $t$, the edge $(i,j)$ is selected, then
\begin{equation}\label{evolution-automaton-ppp}
(n_i(t+1), n_j(t+1))=\left(\giv(n_i(t), n_{j}(t)), \giv(n_i(t),
n_{j}(t))\right).
\end{equation}

The following result characterizes the convergence properties of~\eqref{evolution-automaton-ppp}. Recall that in \eqref{eq:SetR6bis} we defined the set $\A=\setdef{y\in
\integers^N}{\exists \,\alpha\, \in\, \integers \mbox{ such that } y=2\alpha\1}$.

\begin{theorem}\label{th:naivedet2}
Let $n(t)$ evolve according to (\ref{evolution-automaton-ppp}). For every fixed initial condition
$n(0)$, almost surely there exists $\Tcon\,\in\,\integernonnegative$ such that
$n(t)\, \in\,\A$ for all $t \geq \Tcon.$
\end{theorem}
\begin{proof}
The proof is similar to the proof of Theorem~\ref{th:naivedet} and Theorem~\ref{theo:ConvInM}, and it is based on proving the following three facts:
\begin{enumerate}
\item [(i)] each element in the set $\A$ is invariant
 for the evolution described by~\eqref{evolution-automaton-ppp};
 \item [(ii)] $n(t)$ is a Markov process on a finite number of
states; \item [(iii)] starting from any state in $\Z^N$, there is a positive probability for $n(t)$
to reach a state in $\A$ in a finite number of steps.
\end{enumerate}

%Then we have the following result whose proof, even if the structure of (\ref{evolution-automaton-pp}) is different
%from the structure of (\ref{evolution automaton
%p})\footnote{Observe that, if the initial condition $x(0)$
%(\ref{evolution-automaton-p}) describes the evolution of a
%deterministic process, while, also in this case, (\ref{evolution
%automaton pp}) continues to describe the evolution ofd a
%stochastic process.}, follows directly from Theorem~%\ref{th:naivedet}.
%\begin{theorem}\label{th:naivedet2}
%Let $n(t)$ be as above and \eqref{evolution-automaton-ppp} be the
%evolution law. Then almost surely there exists $\Tcon\in \naturals$ and
%$\alpha \in \integers$ such that $n_i(t)=2 \alpha$ for all $i\in  \V$ and
%for all $t\ge \Tcon$.
%\end{theorem}
%\begin{proof}
%The proof is similar to the proof of Theorem~\ref{th:naivedet} and
%Theorem~\ref{theo:ConvInM}, and it is based again on proving the
%following three facts:
%\begin{enumerate}
%  \item The evolution process is a Markov process with a finite number of states;
%  \item The process~\eqref{evolution-automaton-ppp} has absorbing states;
%  \item There is a positive probability of reaching an absorbing state in a finite time, starting for any initial state.
%\end{enumerate}
Let us now check them in order.
%Given $n(t)$ let $m(t)$ and $M(t)$
%be defined as in~\eqref{eq:m-M}.
\begin{enumerate}

\item [(i)] is trivial.

\item [(ii)] Markovianity immediately follows from the fact that
subsequent random choices of the edges are independent and
from~\eqref{evolution-automaton-ppp}. To prove that the states are
finite, define $m(t)$ and $M(t)$ as in~\eqref{eq:m-M}. Let $h\,\in\,\integers$. Then, from the structure of
$\giv$ we have that
  \begin{itemize}
  \item $\giv(h,h)=h$ if $h$ is even;
    \item $h-1\leq \giv(h,h)\leq h+1$ if $h$ is odd.
  \end{itemize}
The above two properties imply that $m(0)-r_{m(0)}\le n_i(t)\leq
M(0)+r_{M(0)}$ for all $i \in \V$ and for all $t\ge 0$. This yields
(ii).
\item [(iii)] Observe that
  $$
  \giv(h,k)=\giii(h,k)-\xi_1 r_h-\xi_2 r_k,
  $$
  where $\giii$ is the map defining the evolution of~\eqref{eq:DinGloQuaStra}.
  Hence
  $$
\Pr\left[\giv(h,k)=\giii(h,k)\right]\geq \frac{1}{4}.
$$
This fact, combined with the fact (iii) proved along the proof of
Theorem~\ref{th:naivedet}, ensures that, also for~\eqref{evolution-automaton-ppp}, there is a positive probability of reaching a
state in $\A$ in a finite time.
\end{enumerate}
\end{proof}
The above theorem and the previous remarks about $\Tall$ lead to the following claim about the original system.
\begin{corollary}\label{corol:tot-det}
Let $x(t)$ evolve following (\ref{eq:GossipQuant1p}). Then almost
surely there exists $\Tcon\in \integernonnegative$ and $\alpha\in \integers$ such that
$x_i(t)=\alpha$, for all $i\in \V$ and for all $t\ge \Tcon$.
\end{corollary}

As for (\ref{eq:GossipQuant1d}), this strategy does not preserve the initial average, in general. However, using the probabilistic quantizer guarantees that the average is preserved at least in expectation, and moreover, the convexity argument developed in the above proof, step (ii), provides a worst case bound on the error committed by the algorithm, as
\begin{equation}\label{eq:BoundOnDeviation2}
|\alpha-\xave(0)|\le \max_{i}{x_i(0)}-\min_{i}{x_i(0)}+2.\end{equation}

\begin{remark}[Deviation from the initial average]\label{rem:deviation}\normalfont
Let us discuss the issue of the deviation from the initial average in the totally quantized algorithm. The worst case bounds~\eqref{eq:BoundOnDeviation} and~\eqref{eq:BoundOnDeviation2} leave the possibility that the amplitude of such deviation depend on the initial condition, but one could expect that the average behavior of the algorithm be better than that. In Figure~\ref{fig:Comparison} we compare the totally quantized strategy, using $\qp$ and $\qd$, in terms of the deviation $z=|\alpha-N^{-1}\1^*x(0)|$. We plot $z$ as a function of $N$, using a sequence of complete graphs of increasing size. For our plot we considered complete graphs, because of  their faster convergence: however, the results are qualitatively the same for other families of graphs (rings, lattices, random geometric graphs). Simulations show that the consensus point obtained using $\qp$ is close to the average of the initial condition, even if the algorithm preserves the average of the states only in expectation. Instead, $\qd$ leads to a consensus point whose distance from the average of the initial condition is larger, increases with $N$, and depends on the initial condition. This behavior can be understood as a consequence of the accumulation of rounding errors. Indeed, the application of the symbolic dynamics $\giii$ increases the sum of the symbolic states by one, {\em every time an odd argument is involved}. Then, the larger the graph, the more time steps are needed for convergence, the more errors accumulate. The randomness due to $\qp$, instead, rules out this effect\footnote{A similar compensation of the accumulation of errors can also be obtained by a deterministic modification of the quantizer $\qd$. For instance, let $\qdtilde$ be defined as $\qd$ except that, for any integer $y$, $\qdtilde(y +0.5) = \begin{cases}y &\text{ if $y$ is odd}\\ y+1 &\text{ if $y$ is even.} \end{cases}$\\ Simulations suggest that such a correction actually reduces the perturbation of the average, but is less effective than randomization.}, and indeed \eqref{eq:GossipQuant1p} preserves the average in expectation at each time step.
\end{remark}

\begin{figure}[ht]
\begin{center}
\includegraphics[width=.52\columnwidth]{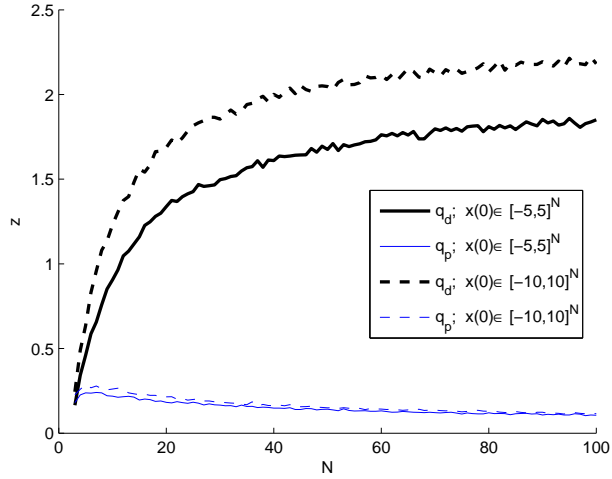}
  \caption{Comparison in terms of $z$ between the totally quantized algorithm using deterministic and probabilistic quantizers, for complete graphs of size $N$. The plotted values are the average of 1000 runs. The initial condition were drawn from a uniform distribution on the given intervals. The plotted lines are not smooth, due to significant variances. For $N=100$, in the four cases of the figure the standard deviation is equal to 0.41, 0.081, 0.50, 0.085, respectively.}   \label{fig:Comparison}
  \end{center}
\end{figure}

%%%%%%%%%%%%%%%%%%%%%%%%%%%%%%%%
\section{Partially quantized update}\label{sec:Naive}
In this section we study the partially quantized update rule defined in~\eqref{eq:GossipQuant3} as
\begin{align*}
x_i(t+1)&= \frac{1}{2} x_i(t)+\frac{1}{2} \hat{x}_j(t)\nonumber\\
x_j(t+1)&= \frac{1}{2} x_j(t)+\frac{1}{2} \hat{x}_i(t).
\end{align*}

\subsection{Deterministic quantizer}
Observe that the update~\eqref{eq:GossipQuant3} does not preserve the average of states. Nevertheless,
if the deterministic quantizer is used, such an update induces the same symbolic dynamics $\gii$ as the compensating strategy.
Indeed, the continuous dynamics is
$$x_i(t+1)=\frac{1}{2}x_i(t)+\frac{1}{2} \qd(x_j(t)),$$
and we deduce that \begin{align*}
\lfloor 2 x_i(t+1)\rfloor=&  \lfloor  x_i(t)\rfloor+ \qd(x_j(t))\\
n_i(t+1)=& \left\lfloor\frac{n_i(t)}{2}\right\rfloor+\left\lceil\frac{n_j(t)}{2}\right\rceil,
\end{align*}
provided $n_i(t)=\lfloor 2 x_i(t)\rfloor$ as above in this paper.
We argue a result which is similar to Corollary~\ref{corol:SymbConv}, but allows a larger deviation from the initial average.
\begin{theorem}\label{thm:ConvNaiveD}
Consider the algorithm \eqref{eq:GossipQuant3} with quantizer $\qd$. Then, there exists $\Tcon\in \integernonnegative$ such that, for all $t\ge\Tcon,$
\begin{equation}\label{eq:conv-det-part}
   {|x_i(t)-x_j(t)|}\leq 1\quad \forall\,i,j\in \until{N}
\end{equation}
and
$\|x(t)-\xave(0)\1 \|_\infty\leq 2.$
\end{theorem}
\begin{proof}
The inequality~\eqref{eq:conv-det-part} can be derived as \eqref{eq:conv-det-comp}, since the symbolic dynamics is the same in both cases. Moreover, since for any integer $k$, it holds that $\lfloor k/2 \rfloor+ \lceil k/2 \rceil=k$, we have that, given the definitions above, $n_i(t+1)+n_j(t+1)=n_i(t)+n_j(t)$ for any time $t$ and any chosen pair $(i,j)$. Hence, the symbolic dynamics preserves the average of the symbolic states, even if the continuous dynamics does not. We then define $\nave(t)=\frac{1}{N}\sum_{i=1}^{N}n_i(t)$, and remark that for any $t\ge0$,
\begin{align*}
\left|\frac{1}{2}\nave(t)-\xave(t)\right|=& \left|\frac{1}{2}\frac{1}{N}\sum_{i=1}^{N}n_i(t)-\frac{1}{N}\sum_{i=1}^{N}x_i(t)\right|\\
=&\frac{1}{N}\left|\sum_{i=1}^{N}\left(\frac{1}{2}n_i(t)-x_i(t)\right)\right|\\
\le&\frac{1}{N}\sum_{i=1}^{N}\left|\frac{1}{2}n_i(t)-x_i(t)\right|\\
=&\frac{1}{N}\sum_{i=1}^{N}\left|\frac{1}{2}\lfloor 2x_i(t)\rfloor-x_i(t)\right|\\
=&\frac{1}{N}\sum_{i=1}^{N}\frac{1}{2}\left|\lfloor 2x_i(t)\rfloor-2x_i(t)\right|\\
\le& \frac{1}{2},
\end{align*}
since $\nave(t)=\nave(0)$. Hence, for $t\ge \Tcon$,
\begin{align*}|\xave(t)-\xave(0)|\le
&\left|\xave(t)-\frac{1}{2}\nave(t)\right|+\left|\frac{1}{2}\nave(t)-\frac{1}{2}\nave(0)\right|+\left|\frac{1}{2}\nave(0)-\xave(0)\right|
\le 1.\end{align*}
This in turn implies, together with \eqref{eq:conv-det-part}, that $\|x(t)-\xave(0)\1 \|_\infty\leq 2.$
\end{proof}

\subsection{Probabilistic quantizer}
The symbolic analysis is again useful to provide a convergence result when the probabilistic quantizer is used, and then~\eqref{eq:GossipQuant3} becomes
\begin{equation}\label{eq:GossipQuant3p}
x_i(t+1)=\frac{1}{2}x_i(t)+\frac{1}{2} \qp(x_j(t)).\end{equation}

\begin{theorem}\label{th:part-qp-both}
Consider the algorithm \eqref{eq:GossipQuant3p}. Then, %there exists $\Tcon\in \integernonnegative$ such that almost surely ${|x_i(t)-x_j(t)|}\leq 1$ for all $i,j\in \until{N}$ and for all $t\geq \Tcon$. Moreover,
it exists an integer random variable $\alpha$, such that $\Exp[\alpha]=\xave(0)$, and almost surely
$$\lim_{t\to\infty}x_i(t)=\alpha,$$
for all $i\in\until{N}$.
\end{theorem}

\begin{proof}
We start our analysis by observing the following fact. Let $x\in \real$, and $n=\lfloor 2 x \rfloor$. If $x-\lfloor x\rfloor<1/2$, then
\begin{equation*}
\qp(x)=
\begin{cases}
\left\lceil\frac{n}{2}\right\rceil &\text{ with probability $\left\lceil x \right\rceil -x$}\\
\left\lceil\frac{n}{2}\right\rceil+1  &\text{ with probability $x-\left\lfloor x \right\rfloor$}
\end{cases}
\end{equation*}
otherwise, if $x-\lfloor x\rfloor\ge 1/2$, then
\begin{equation*}
\qp(x)=
\begin{cases}
\left\lceil\frac{n}{2}\right\rceil &\text{ with probability $x-\left\lfloor x \right\rfloor$}\\
\left\lceil\frac{n}{2}\right\rceil -1  &\text{ with probability $\left\lceil x \right\rceil -x$}
\end{cases}
\end{equation*}
Let now $n_i(t)$, for $i\in \until{N}$, be equal to $n_i(t)=\lfloor 2x_i(t)\rfloor$.
Recalling that $\lfloor x_i(t)\rfloor=\left\lfloor \frac{n_i(t)}{2} \right\rfloor$, from the observations above and \eqref{eq:GossipQuant3p}, it follows that the vector $n(t)$ satisfies the following recursive equation
\begin{align*}
n_i(t+1)&=\left\lfloor\frac{n_i(t)}{2}\right\rfloor+\left\lceil\frac{n_j(t)}{2}\right\rceil + e_j(t),\\
n_j(t+1)&=\left\lfloor\frac{n_j(t)}{2}\right\rfloor+\left\lceil\frac{n_i(t)}{2}\right\rceil + e_i(t),
\end{align*}
where
\begin{equation*}
e_k(t)=\begin{cases}
\begin{cases} 1 &\text{ with probability $x_k(t)-\lfloor x_k(t)\rfloor$}\\
0 &\text{ with probability $1-(x_k(t)-\lfloor x_k(t)\rfloor)$}
\end{cases}
& \text{if $n_k(t)$ is even} \\
\begin{cases}
-1 &\text{ with probability $1-(x_k(t)-\lfloor x_k(t)\rfloor)$}\\
0 &\text{ with probability $x_k(t)-\lfloor x_k(t)\rfloor$}
\end{cases}
& \text{if $n_k(t)$ is odd.}
\end{cases}
\end{equation*}
By reintroducing the map $\gii$ defined in \eqref{evolution-automaton-p}, we can write in a compact way that
\begin{align*}
(n_i(t+1),n_j(t+1))=\gii(n_i(t),n_j(t))+(e_j(t),e_i(t)).
\end{align*}
Note that $n_k(t)$ is even if and only if $x_k(t)-\lfloor x_k(t)\rfloor<1/2$. Then, we have that $\Pr(e_j(t)= 0)\geq \frac{1}{2}$, and hence $\Pr\left[(n_i(t+1),n_j(t+1))=\gii(n_i(t),n_j(t))\right]\ge \frac{1}{4}.$
Moreover, observe that, recalling the notations in the proof of Theorem~\ref{th:naivedet}, for all $i\in \until{N}$ and $t>0$,
\begin{equation*}%\label{eq:boundDS}
m(0)-r_{m(0)}\le n_i(t)\le M(0)+1-r_{M(0)}.\end{equation*}
This invariance property and the above connection with $\gii$ imply that almost surely it exists $\Tcon\in\integernonnegative$ such that $n(t)$ belongs to the set
$\A=\setdef{y\in \integers^N}{\exists \,\alpha\, \in\, \integers \mbox{ such that } y=2\alpha \1 }$
defined in \eqref{eq:SetR6bis}.

To conclude the proof, we remark that, for any $t\ge0$, and any $i\in\until{N}$,
$$2x_i(t+1)-\lfloor2 x_i(t+1)\rfloor=x_i(t)+\qp(x_j(t))-\lfloor x_i(t)+\qp(x_j(t))\rfloor
=x_i(t)-\lfloor x_i(t)\rfloor.$$
Equivalently, $2x_i(t+1)-n_i(t+1)=x_i(t)-\lfloor \frac{n_i(t)}{2}\rfloor.$
If now we assume $t\ge \Tcon$, then we have shown that almost surely $n_i(t+1)=n_i(t)=2\alpha$, for all $i\in\until{N}.$ This fact implies that the equality above reduces to $2x_i(t+1)-2\alpha=x_i(t)-\alpha$, which in turns yields that
\begin{equation*}
x_i(t+1)=\frac{1}{2}(x_i(t)+\alpha).
\end{equation*}
The latter deterministic dynamics implies that $\lim_{t\to\infty}x_i(t)=\alpha$, for all $i\in\until{N}.$ The argument is concluded recalling that the average is preserved in expectation.

\end{proof}

%To prove the second assertion, we define the function $\map{S}{\reals}{[0,1)}$ as $S(x)=x-\lfloor x \rfloor$, and assume $t>\Tcon$. Then
%\begin{equation*}
%S(x_i(t+1))=\frac{1}{2} S(x_i(t))+\frac{1}{2}\xi_j(t),\end{equation*}
%where $\xi_j(t)$ is a Bernoulli random variable with expectation $S(x_j(t))$.
%Let now $R(t)=\sum_{l=1}^{N}S(x_l(t)).$ Then,
%\begin{align*}
%\Exp[R(t+1)]=&\sum_{k\neq i,j} S(x_k(t))+\frac{1}{2} \left(S(x_i(t))+S(x_j(t))\right)+\frac{1}{2}(\Exp[\xi_i(t)]+\Exp[\xi_j(t)])\\
%=&R(t).\end{align*}
%Since $0\le R(t)\le N$, then $R(t)$ is a bounded martingale, and by Doob's Theorem it converges almost surely, that is $R(t)\to \rho$ a.~s., $\rho\in [0,N].$
%Recall that the state vectors in the form $x(t)=\beta \1$, with $\beta\in \integer$, are the only a.~s. equilibria of~\eqref{eq:GossipQuant3p}, in the sense that if $x(t)=\beta \1$, then almost surely $x(t+1)=x(t).$ This allows to argue that actually $\rho\in \{0,N\}$ and $x(t)\to \beta\1$ almost surely.

%%%%%%%%%%%%%%%%%%%%%%%%%%%%%%%%%%%%%%%%%%%%%%%%%%%%%%%%%%%%%%%%%%%%%%%%%%%%%%%%%%%%%%%%%%%%%%%%%%%%%%%%%%%%%%%%%%%
%
%%%%%%%%%%%%%%%%%%%%%%%%%%%%%%%%%%%%%%%%%%%%%%%%%%%%%%%%%%%%%%%%%%%%%%%%%%%%%%%%%%%%%%%%%%%%%%%%%%%%%%%%%%%%%%%%%%%
\section{Discussion}\label{sec:conclusions}

In this paper we studied the gossip algorithm for the consensus problem with quantized communication. We proposed three update rules, a {\it totally quantized}, a {\it partially quantized} and a {\it compensating} rule. In the first one, the agents use only quantized information in order to update their states. In the second one, they have access to exact (non-quantized) information regarding their own state, and use it for the update. In the third one, each agent uses both exact and quantized information about its own state, in order to ensure that the average of the states is globally preserved. This positive feature has the negative counterpart of preventing that consensus (all states be equal) be exactly reached. In all the algorithms we proposed, consensus can be approached up to the quantizer precision. Moreover, deviations from the average of the initial state are small, except in the case of the totally quantized rule with deterministic quantizer. We summarize these results in Table~\ref{table:Summary}. In the remaining of this section, we discuss some extensions and potential developments of our results.
\begin{table}[ht]
\begin{tabular}{|c|c|c|c|}
\hline
   & Totally Quantized & Compensating & Partially Quantized \\ [0.5ex]
\hline
 $\qd$ & \begin{tabular}{c}  Finite time conv. to consensus \\ Large error from average    \end{tabular} &
 \begin{tabular}{c}  Finite time conv. to \\ $\|x-\xave(0)\1\|_\infty\le 1/2$ \\ Average preserved   \end{tabular}
 & \begin{tabular}{c}  Finite time conv. to \\ $\|x-\xave(0)\1\|_\infty\le 2$ \\ Average not preserved   \end{tabular} \\
\hline
$\qp$
  &
  \begin{tabular}{c} Finite time conv. to consensus \\ Average preserved in expectation   \end{tabular}
  &
  \begin{tabular}{c} Asympt. conv. to \\ $N^{-1/2}\sqrt{\Exp[\|x-\xave(0)\1\|^2_2]}\le 1/2$\\ Average preserved \end{tabular}
  &
  \begin{tabular}{c} %Finite time conv. to \\ $N^{-1/2}\sqrt{\Exp[\|x-\xave(0)\1\|^2_2]}\le 1/2$\\
   Asympt. conv. to consensus \\ Average preserved in expectation \end{tabular}
  \\ [1ex] % [1ex] adds vertical space
\hline
\end{tabular}
\caption{Summary of results.}
\label{table:Summary}
\end{table}
\subsection{Uniform quantizers}
In Section~\ref{sec:Statement} we introduced two significant quantizers, $\qd$ and $\qp$. The present section provides a brief discussion about the generality of our results and about the motivations for our choices. Here we shall consider a class of uniform quantizers, that is maps $\map{q}{\reals}{\integers}$ with the property that $q(x)\in\{\lfloor x\rfloor,\lceil x\rceil\}$.

Let us first consider {\em deterministic} uniform quantizers.
The results given for $\qd$, namely Corollaries~\ref{corol:SymbConv} and~\ref{cor:ConvPartiallyD}, and Theorem~\ref{thm:ConvNaiveD}, can be extended to the {\em rounding down} $x \mapsto\lfloor x\rfloor$ and the {\em rounding up} quantizer $x\mapsto \lceil x\rceil$, defining suitable symbolic dynamics on $n(t)$. Let us consider the rounding down, first. For the compensating strategy, the induced symbolic dynamics is
$$\gvi(h,k)=\left(\left\lfloor\frac{k}{2}\right\rfloor+\left\lceil\frac{h}{2}\right\rceil,\left\lfloor\frac{h}{2}\right\rfloor+\left\lceil\frac{k}{2}\right\rceil\right),$$
whereas for both the totally and the partially quantized strategy, it is
$$\gvii(h,k)=\left\lfloor\frac{h}{2}\right\rfloor+\left\lfloor\frac{k}{2}\right\rfloor.$$
Their analysis is analogous to the one in Theorems~\ref{th:naivedet} and~\ref{theo:ConvInM}, respectively.
The analysis of the rounding up quantizer easily follows if we notice that $\lceil x\rceil=-\lfloor -x\rfloor.$

Second, we consider {\em randomized} uniform quantizers, which generalize the probabilistic quantizer $\qp$. We observe that the proof of Theorem~\ref{th:naivedet2} is based on the fact that for all $a\in \Z$,
\begin{equation}\label{eq:QuantFamily}
\qp(a+1/2)=\left\{
\begin{array}{cc}
a+1 & \mbox{ with probability } p \\
a & \mbox{ with probability } 1-p,
\end{array}\right.
\end{equation}
with $p=1/2$. Actually, the argument is valid for any $p\in (0,1),$ and hence Corollary~\ref{corol:tot-det} can be extended to any uniform quantizer satisfying Equation~\eqref{eq:QuantFamily} with $p\in(0,1).$
On the other hand, the mean squared error results in Section~\ref{sec:PartQA-qp} are a direct consequence of the statistical properties of $\qp$ stated in Lemma~\ref{lem:PropQuantProb}: they thus hold for any quantizer sharing these properties. In particular, both for Section~\ref{sec:PartQA-qp} and for Theorem~\ref{th:part-qp-both}, the randomized quantizer needs to be unbiased. The probabilistic quantizer $\qp$ is the {\em only} unbiased uniform quantizer, and hence it is the only one to which all our arguments apply. This justifies our choice of the specific probabilistic quantizer $\qp$.

\subsection{Weighted averaging}
For simplicity reasons we have chosen to consider consensus gains equal to $1/2$. Hence, although we believe that the case we studied already shows the significant features of the problem, it is natural to ask whether our results can be extended to general choices of weights, considering an update matrix
\begin{equation}\label{eq:eps-update}P(t)=I-\eps E_{ij},\end{equation} instead of the matrix~\eqref{eq:defPt}.
Such extension is likely to be possible, modulo solving some technical difficulties in the definition and analysis of the suitable symbolic dynamics. For instance, we conjectured in \cite{PF-RC-FF-SZ:08cdc} that Corollary~\ref{corol:SymbConv} apply to the compensating rule with deterministic quantizers and general update \eqref{eq:eps-update}, provided $\eps$ is a rational number in $(0,1/2]$. This fact has actually been proven for any $\eps\in (0,1/2]$ in the pair of recent conference papers \cite{JL-RM:09a,JL-RM:09b}.

\subsection{Speed of convergence}
Our paper has been devoted to prove convergence results for a set of algorithms. After proving convergence, the second analysis issue is studying the speed of convergence of the algorithms or, equivalently, the time needed to reach (or to approach in a suitable sense) their limit states. The non-quantized gossip algorithm \cite{SB-AG-BP-DS:06} is known to asymptotically converge (in a mean squared sense) at exponential speed, with a rate which depends on the matrix $W.$ It is thus natural to conjecture that the convergence of the quantized version be roughly exponential, as long as the differences among states are much larger than the quantization step: results in this sense have been given in \cite{PF-RC-FF-SZ:08cdc}. This belief is confirmed by simulations, as those we report in Figure~\ref{fig:RatesComparison}.

\begin{figure}[ht]
\begin{center}
\includegraphics[width=.49 \textwidth]{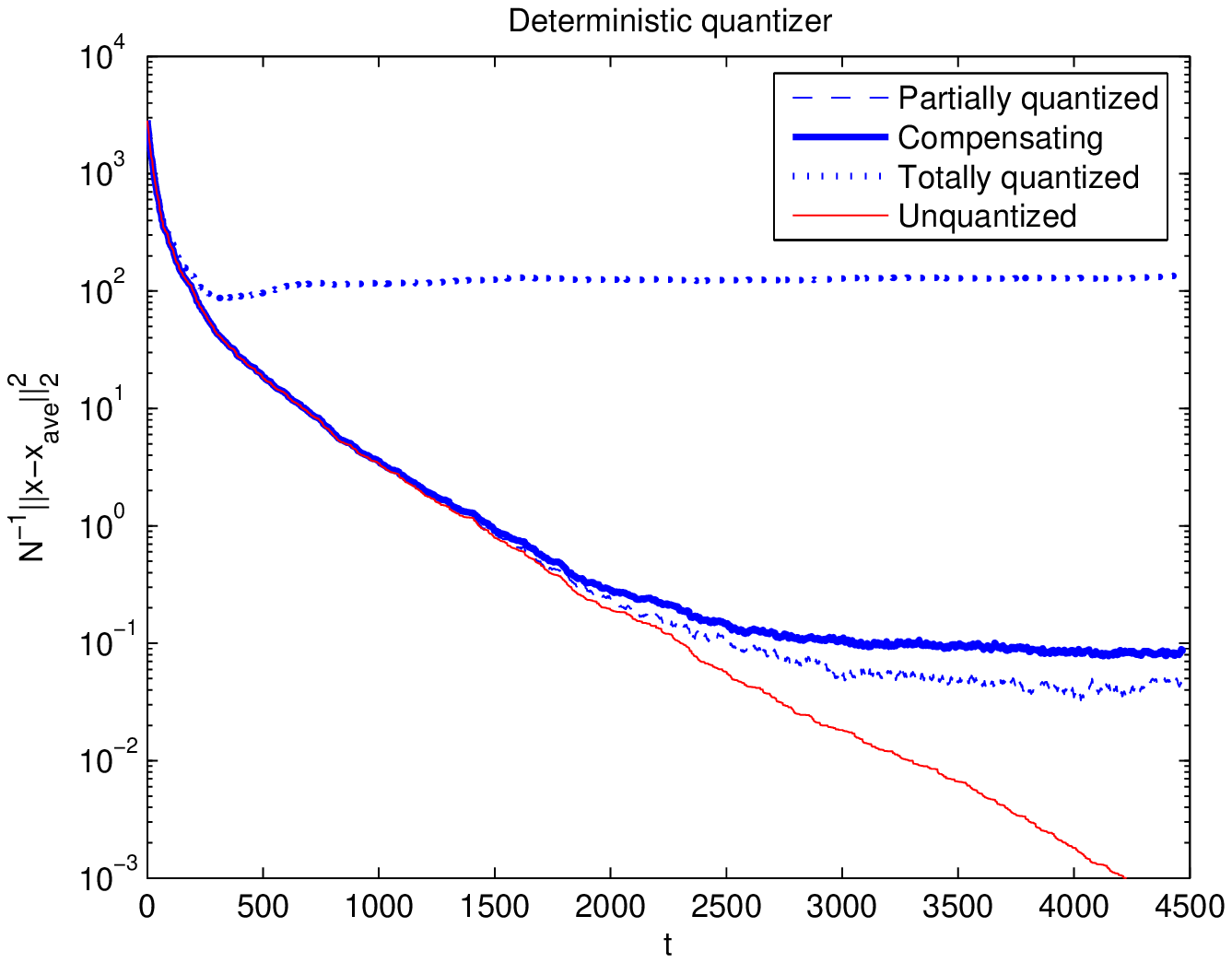}
\includegraphics[width=.49 \textwidth]{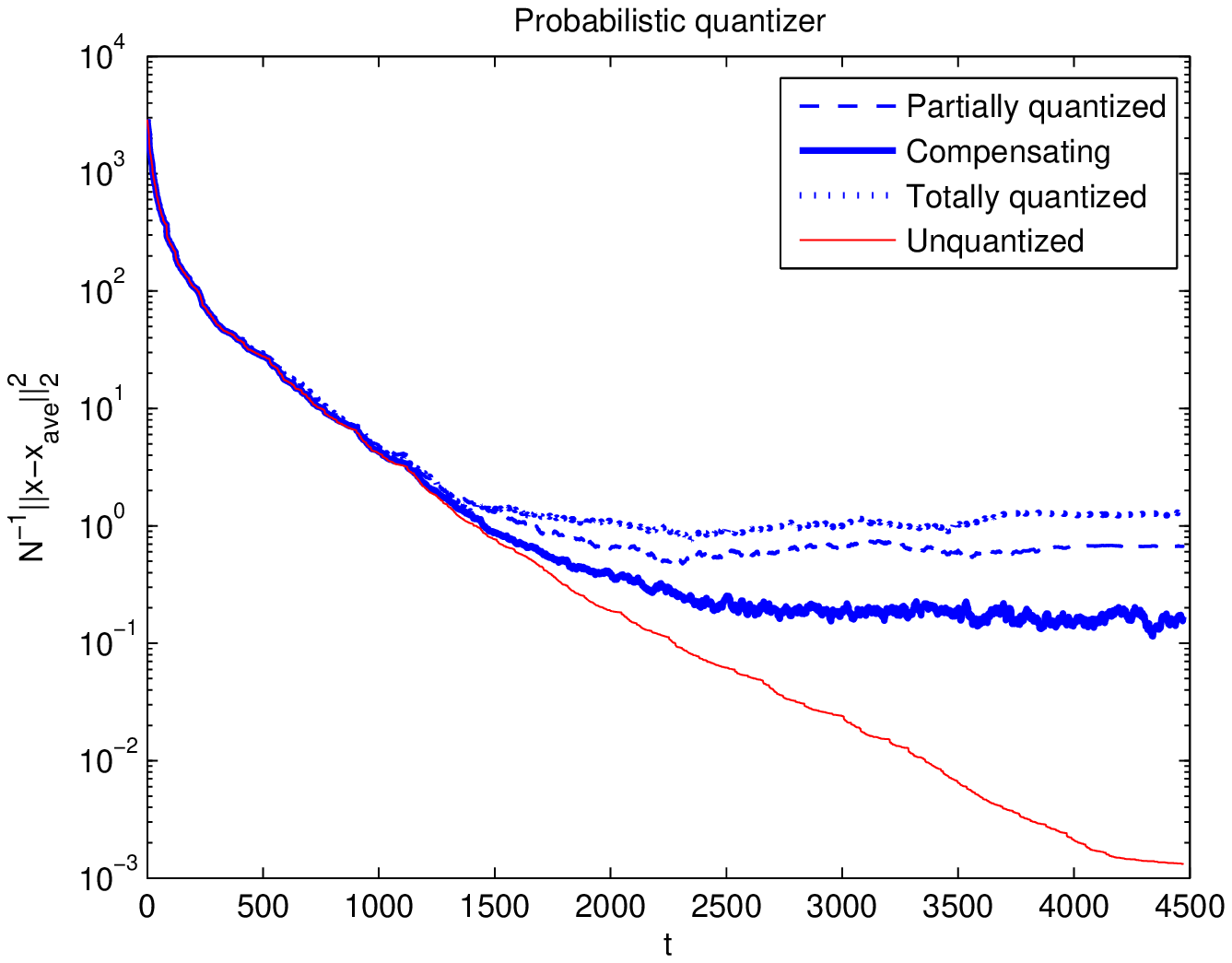}
  \caption{Evolution of the squared distance from the initial average for the three algorithms presented in the paper, compared with the standard gossip algorithm. Average of 10 runs on a geometric graph of $N=20$ nodes. Random uniform initial condition with $x_i(0)\in [-100,100]$ for all $i\in \until{N}.$  }
 \label{fig:RatesComparison}\end{center}
 \end{figure}

However, as we have shown in this paper, the granularity effects eventually come out in the convergence, making the systems converge in finite time to some limit set. Upper bounds on this convergence time can be found by the Markov chains techniques used in \cite{AK-TB-RS:07}, and recently in \cite{JL-RM:09b} and \cite{MZ-SM:08a}. Such bounds are usually very conservative in terms of their dependence on the number of agents: for this reason, we do not pursue such analysis here.
In our opinion, it would be interesting to better understand the transition in the behavior of the algorithm: what is missing is a sharp result quantifying {\em when} the granularity effects become non-negligible, and which is the speed of convergence of the system during such transition.

\bibliographystyle{plain}
\bibliography{aliasFrasca,RefFrasca,PF}

\end{document}